\documentclass[a4paper,Haag duality]{mathscan}
\newtheorem{thm}{Theorem}[section] 
\newtheorem{pro}[thm]{Proposition}

\theoremstyle{definition}           
\newtheorem{defn}[thm]{Definition}  

\newcommand{\NI}{\noindent}

\newcommand{\bea}{\begin{eqnarray}}
\newcommand{\eea}{\end{eqnarray}}

\def \b #1 {\bf #1}
\newcommand{\IR}{\mathbb{R}}

\newcommand{\IC}{\mathbb{C}}

\newcommand{\IT}{\mathbb{T}}
\newcommand{\IZ}{\mathbb{Z}}
\newcommand{\cal}{\mathcal}
\newcommand{\clk}{{\cal K}}

\newcommand{\cla}{{\cal A}}

\newcommand{\cli}{{\cal I}}

\newcommand{\clh}{{\cal H}}
\newcommand{\clp}{{\cal P}}
\newcommand{\clo}{{\cal O}}
\newcommand{\clb}{{\cal B}}

\newcommand{\clj}{{\cal J}}

\newcommand{\cld}{{\cal D}}

\newcommand{\clm}{{\cal M}}

\newcommand{\raro}{\rightarrow}

\newcommand{\vsp}{\vskip 1em}

\newcommand{\ul}{\underline}

\def \qed {\hfill \vrule height6pt width 6pt depth 0pt}
\newcommand{\be}{\begin{equation}}
\newcommand{\ee}{\end{equation}}
\newcommand{\ben}{\begin{eqnarray*}}
\newcommand{\een}{\end{eqnarray*}}
\begin{document}

\title{Translation invariant pure state on $\clb=\otimes_{j \in \IZ}\!M^{(j)}_d(\IC)$ and its split property}

\author{ Anilesh Mohari }
\thanks{...}

\address{ The Institute of Mathematical Sciences, CIT Campus, Taramani, Chennai-600113 }

\email{anilesh@imsc.res.in}



\thanks{ }

\begin{abstract}
We prove that a real lattice symmetric reflection positive translation-invariant pure state of $\clb=\otimes_{j \in \IZ}M^{(j)}_d(\IC)$ admits split property, if and only if its two-point spatial correlation functions decay exponentially. 
\end{abstract}

\maketitle 
\section{ Introduction }

\vsp 
Translation invariant pure states often appear as ground states of the two-side infinite quantum chain Heisenberg type of models [BR-II,Si]. Several types of cluster properties for two-point spatial correlation functions of pure states are investigated in the last few decades, both by operator algebraist and mathematical physicist in order to understand its mathematical structure besides its physical implications on bulk parameters in low temperature condense matter physics. In this paper our primary concern is to relate two apparently different cluster conditions known in the literature as `split property' and `exponentially decaying property' for two-point correlation functions of a translation-invariant state. These two properties are strongly related, as both ask for a decay of two-point correlation functions, however defined using different quantifiers. That is, `the split property' demands that the rate of the decay is independent from the observables, uniformly in norm, but arbitrary. In comparison, the exponentially decaying of two-point spatial correlation function only demands the decay exponent be independent of observables, but total decay as such may not be independent of two concern observables. Thus in some sense, both the cluster conditions have some universal properties with respect to observables but expressed quite differently.  

\vsp 
There is also another characterization ([Ma3] using Theorem 2.7 in [Pow]) of the split property in terms of the type of the von Neumann algebra generated by one half of the chain that it has to be type I. A recent result [Mo1] also says that the type of the von Neumann algebra generated by one half of the chain has to be either a type-I or a type-III factor for 
a two-sided translation-invariant pure state. Thus the split or non split property of a two-sided translation-invariant state is clearly related to which type of von-Neumann algebra 
are generated by the one half of the infinite chain. On the other hand, `the exponential decaying condition' has an easier appeal for numerical analysis, besides an interesting quantity for experimental labs for low temperature physics of magnetic materials with lattice or quasi lattice structure [DR]. Some recent development in quantum information theory also aimed to understand entanglement of a translation-invariant pure state $\omega$ on $\clb=\otimes_{j \in \IZ}\!M_d^{(j)}(\IC)$ [KMSW] by treating it as a state on the bipartite systems consisting of two half infinite chains and its deeper relation to Bell's inequality [SW] which is a functional of two-point spatial correlation functions of $\omega$. In other words two-point spatial correlation functions plays a crucial role to determine bulk quantities such as magnetic moments and maximal violation of Bell's inequality.
  
\vsp 
It is the goal of this paper to connect these two properties. We prove that translational invariant pure states on the two-side infinite quantum spin chain, further satisfying a few additional symmetry assumptions, possess the split property, if and only if its two-point correlation functions decay exponentially. Below we formulate the problem in a little more general framework of $C^*$-dynamical system [BR-II] namely for higher dimensional 
lattices of quantum spin chains.    

\vsp 
The uniformly hyper-finite $C^*$-algebra $\clb=\otimes_{\ul{j} \in \IZ^k}\!M_d^{(\ul{j})}(\IC)$ of infinite tensor product of $d \times d$-square matrices $\!M^{(\ul{j})}_d(\IC) \equiv \!M_d(\IC)$ levelled by $k-$dimensional lattice points $\ul{j}=(j_1,j_2..,j_k) \in \IZ^k\;,\;k \ge 1$ of integers, is the norm closure of algebraic inductive limit with the 
net of finite dimensional $C^*$ algebras $\clb_{\Lambda}= \otimes_{\ul{j} \in \in \Lambda }\!M_d^{(\ul{j})}(\IC)$, where $\Lambda \subset \IZ^k$ are finite subsets and an element in $\clb_{\Lambda_1}$ is identified with element in $\clb_{\Lambda_2}$ by the inclusion map if $\Lambda_1 \subseteq \Lambda_2$. We use the symbol $\clb_{loc}$ to denote union of all local algebras $\clb_{\Lambda}: \Lambda \subset \IZ^k$ finite subsets. Thus $\clb$ is a quasi-local $C^*$-algebra with union of local algebras  $\clb_{loc}$ dense in $\clb$ and $\clb_{\Lambda}'=\clb_{\Lambda^c}$, where $\clb'_{\Lambda}$ is the commutant of $\clb_{\Lambda}$ and $\Lambda^c$ is the complementary set of $\Lambda$ in $\IZ^k$. 
We refer to Chapter 6 of [BR-II] for more details on quasi-local $C^*$-algebras,      

\vsp 
The lattice $\IZ^k$ is a group under co-ordinate wise addition and for each $\ul{n} \in \IZ^k$, we have an automorphism $\theta^{(\ul{n})}$, extending the translation action, which takes $Q^{(\ul{j})}$ to $Q^{(\ul{j}+\ul{n})}$ for any $Q \in \!M_d(\IC)$ and $\ul{j} \in \IZ^k$, by linearity and multiplicative property on $\clb$. A unital positive linear functional $\omega$ on $\clb$ is called {\it state}. It is called {\it translation-invariant} if $\omega = \omega \theta^{(\ul{n})}$ for all $\ul{n} \in \IZ^k$. A linear automorphism or anti-automorphism $\beta$ [Ka] on $\clb$ is called {\it symmetry } for $\omega$ if $\omega \beta = \omega$. Our primary objective is to study in mathematical generality translation-invariant states and their symmetries that find relevance in Hamiltonian dynamics of quantum spin chain models [BR-II,Ru,Si].     

\vsp 
We consider [BR-II,Ru] quantum spin chain Hamiltonian in $k-$dimensional lattice 
of the following form
\be 
H= \sum_{ \ul{n} \in \IZ_k} \theta^{(\ul{n})}(h_0)
\ee
for $h^*_0=h_0 \in \clb_{loc}$, where the formal sum in (1) gives a group of auto-morphisms $\alpha=(\alpha_t:t \in \IR)$ by 
the thermodynamic limit: $\mbox{lim}_{\Lambda_{\eta} \uparrow \IZ_k}||\alpha^{\Lambda_{\eta}}_t(A)-\alpha_t(A)||=0$ for all $A \in \clb$ and $t \in \IR$ for 
a net of finite subsets $\Lambda_{\eta}$ of $\IZ$ with uniformly bounded surface energy, where automorphisms $\alpha^{\Lambda}_t(x)=e^{itH_{\Lambda}}xe^{-itH_{\Lambda}}$ is determined by finite the subset $\Lambda$ of $\IZ_k$ and $H_{\Lambda}=\sum_{\ul{n} \in \Lambda} \theta^{(\ul{n})}(h_0)$. Further the limiting automorphism $(\alpha_t)$ does not depend on the net that we choose in the thermodynamic limit $\Lambda_{\eta} \uparrow \IZ_k$ provided the surface energies of $H_{\Lambda_\eta}$ are kept uniformly bounded. The uniquely determined group of automorphisms $(\alpha_t)$ on $\clb$ is called {\it Heisenberg flows } of $H$. In particular,  we have $\alpha_t \circ \theta^{(\ul{n})} = \theta^{(\ul{n})} \circ \alpha_t$ for all $t \in \IR$ and $\ul{n} \in \IZ^k$. Any linear automorphism or anti-automorphism $\beta$ on $\clb_{loc}$, keeping the formal sum (1) in $H$ invariant, will 
also commute with $(\alpha_t)$.  

\vsp 
A state $\omega$ is called {\it stationary} for $H$ if $\omega \alpha_t= \omega$ on $\clb$ for all $t \in \IR$. The set of stationary states of $H$ is a non-empty compact convex set and is extensively studied in the last few decades within the framework of ergodic theory for $C^*$-dynamical systems [BR-I]. However,  a stationary state of $H$ need not be always translation-invariant. A stationary state $\omega$ of $\clb$ for $H$ is called $\beta$-KMS state at an inverse positive temperature $\beta > 0$ if there exists a function $z \raro f_{A,B}(z)$, analytic on the open strip $0 < Im(z) < \beta$, bounded continuous on the closed strip $0 \le Im(z) \le \beta$ with boundary condition 
$$f_{A,B}(t)=\omega_{\beta}(\alpha_t(A)B),\;\;f_{A,B}(t+i\beta)=\omega_{\beta}(\alpha_t(B)A)$$
for all $A,B \in \clb$. The KMS condition expresses approximate commutation rules for 
two elements $A,B \in \clb$ within the evaluation by $\omega_{\beta}$. Using weak$^*$ compactness of convex set of states on $\clm$, finite volume Gibbs state $\omega_{\beta,\Lambda}$ is used to prove existence of a KMS state $\omega_{\beta}$ for $(\alpha_t)$ at inverse positive temperature $\beta > 0$. The set of KMS states of $H$ at a given inverse temperature forms a non empty simplex and its extreme points are translation-invariant ergodic states. This gives a strong motivation to study translation-invariant states 
in a more general framework of $C^*$-dynamical systems [BR-I].    

\vsp 
A state $\omega$ of $\clb$ is called {\it ground state} for $H$,if the following two conditions are satisfied:

\NI (a) $\omega(\alpha_t(A))=\omega(A)$ for all $t \in \IR$; 

\NI (b) If we write on the GNS space $(\clh_{\omega},\pi_{\omega},\Omega)$ of $(\clb,\omega)$, $$\alpha_t(\pi_{\omega}(A))=e^{itH_{\omega}}\pi_{\omega}(A)e^{ -itH_{\omega}}$$ 
for all $A \in \clb$ with $H_{\omega}\Omega=0$ then, $H_{\omega} \ge 0$.   

\vsp 
By taking low temperature limit of $\omega_{\beta}$ as $\beta \raro \infty$, one also proves existence of a ground state for $H$ [Ru,BR-II]. On the contrary to KMS states, the set of ground states is a convex face in the convex set of $(\alpha_t)$ invariant states of $\clb$ and its extreme points are pure states of $\clb$ i.e. it can not be expressed as a convex combination of two different states of $\clb$. 

\vsp 
On the other hand, many interesting results on ground states, that are known for specific Heisenberg type of models [LSM], such as $XY$ models [Ara2,AMa], Ghosh-Majumder models [GM] and $AKLT$ models [AKLT], give rises interesting conjectures [AL],[Ma3] on the general behaviour of ground states and its physical implication for prime Hamiltonian such as Heisenberg iso-spin $H_{XXX}$ models. In particular,  anti-ferromagnetic Heisenberg iso-spin models find special place in low temperature physics of magnetic materials [Ef,DR,Ma3]. We refer interested readers for an historical account to the survey paper [Na]. 

\vsp 
Further we say a ground state $\omega$ is {\it non-degenerate}, if null space of $H_{\omega}$ is spanned by $\Omega$ only. We say $\omega$ has a {\it mass gap}, if the spectrum $\sigma(H_{\omega})$ of $H_{\omega}$ is a subset of $\{ 0 \} \bigcap [\delta, \infty)$ for some $\delta >0$. For a wide class of spin chain models [NaS], which includes  Hamiltonian $H$ with finite range interaction, $h_0$ being in $\clb_{loc}$, the existence of a non vanishing spectral gap of a ground state $\omega$ of $H$ implies exponential decaying two-point spatial correlation functions. We present now a precise definition for exponential decay of two-point spatial correlation functions of $\omega$. We use symbol $\Lambda^c_m$ for complementary set of the finite volume box $\Lambda_m = \{ \ul{n}=(n_1,n_2..,n_k): -m \le n_j \le m \}$ for $m \ge 1$ 
and $||\ul{n}||^2=\sum_{1 \le j \le k}|n_j|^2$. 

\vsp
\begin{defn} 
Let $\omega$ be a translation-invariant state of $\clb$. We say that the two points spatial correlation functions of $\omega$ {\it decay exponentially}, if there exists a $\delta > 0$ satisfying the following condition: for any two local elements $Q_1,Q_2 \in \clb$ and $\epsilon > 0$, there exists an integer $m \ge 1$ such that    
\be
e^{\delta ||\ul{n}||} |\omega( Q_1 \theta^{(\ul{n})}(Q_2) ) - \omega(Q_1) \omega(Q_2)| \le \epsilon
\ee
for all $\ul{n} \in \Lambda^c_m$.  
\end{defn}

\vsp 
We also recall [BR-II,Ma3] a standard definition of a state to be split in the following. For a more general definition of split property we refer to [DL]. For the present problem, we follow the definition of split property adopted in [Ma3].  

\vsp
\begin{defn} 
Let $\omega$ be a translation-invariant state of $\clb$ and $\omega_{\Lambda}$ be the state $\omega$ restricted to $\clb_{\Lambda}$. We say that $\omega$ is {\it split}, if the following condition is valid for any subset $\Lambda$ of $\IZ_k$: Given any $\epsilon > 0$ there exists a $m \ge 1$ so that
\be
\mbox{sup}_{||Q|| \le 1}|\omega(Q)-\omega_{\Lambda} \otimes \omega_{\Lambda^c}(Q)| \le \epsilon,
\ee
where the above sup is taken over all local elements $Q \in \clb_{ \Lambda^c_m }$ with the norm less than equal to $1$. 
\end{defn}

\vsp 
The uniform clustering property (3) of the state $\omega$ has its mathematical appeal which guarantees that $\omega$ is {\it quasi equivalent } to the tensor product state $\omega_{\Lambda} \otimes \omega_{\Lambda^c}$ by Theorem 2.7 in [Pow]. In contrast to split condition (3), the exponent $\delta > 0$, in the exponentially decaying clustering condition (2), is independent of $Q_1,Q_2 \in \clb_{loc}$, however $m \ge 1$ 
in (2) may depend on $Q_1$ and $Q_2$.   

\vsp 
A state $\omega$ on a $C^*$-algebra $\clb$ is called factor, if the center of the von-Neumann algebra $\pi_{\omega}(\clb)''$ is trivial, where $(\clh_{\omega},\pi_{\omega},\Omega)$ is the GNS space associated with $\omega$ on $\clb$ [BR-I] and $\pi_{\omega}(\clb)''$ is the double commutant of $\pi_{\omega}(\clb)$. A state $\omega$ on $\clb$ is pure, if and only if $\pi_{\omega}(\clb)''=\clb(\clh_{\omega})$, the algebra of all bounded operators on $\clh_{\omega}$. Here we fix our convention that Hilbert spaces that are considered here are always equipped with inner products $ \langle .,. \rangle $ which are linear in the second variable and conjugate linear in the first variable. 

\vsp 
We recall a well known result, Theorem 2.5 in [Pow], that a translation-invariant state $\omega$ of $\clb$ is a factor state,  if and only if for any given $Q_2 \in \clb$ and $\epsilon > 0$, there exists 
an integer  $m \ge 1$ so that  
\be 
\mbox{sup}_{ Q_1 \in \clb_{\Lambda^c_m},||Q_1|| \le 1}|\omega(Q_1 \theta^{(\ul{n})}Q_2)) - \omega(Q_1)\omega(Q_2)| \le \epsilon 
\ee 
for all $\ul{n} \in \Lambda^c_m$.  In particular,  this criteria is used to deduce that a translation-invariant state $\omega$ of $\clb$ is a factor state,  if and only if 
$\omega_{\Lambda}$ ( $\omega_{\Lambda^c}$ ) is a factor state for all subsets of 
$\Lambda$ of $\IZ^k$. In particular,  a translation-invariant split state $\omega$ is a factor state as the uniform clustering condition (3) is a stronger condition than Power's clustering criteria (4). Furthermore, if $\omega$ is a pure translation-invariant state, then 
$\omega_R( \omega_L )$ is type-I,  if and only if $\omega$ is also a split state 
(Proposition 2.3 in [Ma3]). A {\it Gibbs state } [BR-II, Chapter 6.2.2] of a finite range interaction is split. The canonical trace of $\clb$ is a non-pure split state and unique ground state of XY model [AMa,Ma1,Ma2] is a non-split pure state. Our central aim in this paper is to find a criterion for a translation-invariant pure state $\omega$ to be split. 
The UHF algebra $\clb$ being a quasi local $C^*$ algebra, following [Ha, DHR] we say that a translation-invariant pure state $\omega$ of $\clb$ admits {\it Haag duality property }, if 
$$\pi_{\omega}(\clb_{\Lambda^c})'' = \pi_{\omega}(\clb_{\Lambda})'$$ 
for all subsets $\Lambda$ of $\IZ^k$ in the GNS space $(\clh_{\omega},\pi_{\omega},\Omega)$ of $(\clb,\omega)$ and $\Lambda^c$ is the complementary set of $\Lambda$ in $\IZ^k$. A proof for Haag duality is given only for the case $k=1$ in a recent paper [Mo2] but the method used in the proof,  does not seem to have a ready adaptation to higher lattice dimensions. For this reason we confine ourselves now onwards only to the case $k=1$ i.e. one dimensional quantum spin chain since the Haag duality property is crucially used in Proposition 3.2 
for further analysis of pure states.      

\vsp 
Let $\clb=\otimes_{j \in \IZ}\!M^{(j)}_d(\IC)$ be the uniformly hyper-finite $C^*$-algebra over the lattice $\IZ$, where $\!M^{(j)}_d(\IC)$ denote a copy of 
the algebra of $d \times d$-matrices over the field of complex numbers $\IC$ [Sa]. 
In other words $\clb$ is the $C^*$ -completion of the infinite tensor product of the algebra $\!M_d(\IC)$ of $d$ by $d$ complex matrices,  where each component of the tensor product element is indexed by an integer $j$. Let $Q$ be a matrix in $\!M_d(\IC)$. By $Q^{(j)}$ we denote the element $...\otimes 1 \otimes 1 ... 1 \otimes Q \otimes 1 \otimes ... 1\otimes ... $, where $Q$ appears in the $j$-th component. Given a subset $\Lambda$ of $\!Z$, $\clb_{\Lambda}$ is defined as the $C^*$-sub-algebra
of $\clb$ generated by all $Q^{(j)}$ with $Q \in \!M_d(\IC)$, $j \in \Lambda$. We also set
$$\clb_{loc}= \bigcup_{\Lambda:|\Lambda|  \langle  \infty } \clb_{\Lambda}$$
, where $|\Lambda|$ is the cardinality of $\Lambda$. Let $\omega$ be a state on $\clb$. The restriction of $\omega$
to $\clb_{\Lambda}$ is denoted by $\omega_{\Lambda}$. We also set $\omega_{R}=\omega_{[1,\infty)}$ and $\omega_{L}=
\omega_{(-\infty,0]}$. For each $k \in \IZ$, $\theta^k$ is an automorphism of $\clb$ defined by extending the action given by $\theta^k(Q^{(j)})=Q^{(j+k)}$ for all $Q \in \!M_d(\IC)$ and $j \in \IZ$. Thus $\theta^1,\theta^{-1}$ are unital $*$-endomorphisms 
on $\clb_R$ and $\clb_L$ respectively. We use often simply $\theta$ instead of $\theta^1$. 
We say $\omega$ is translation-invariant, if $\omega \circ \theta = \omega$ on $\clb$. In such a case $(\clb_R,\theta,\omega_{R})$ and $(\clb_L,\theta^{-1},\omega_{L})$ are two unital $*$-endomorphisms with invariant states. In this paper we will consider only translation-invariant states of $\clb$. 

\vsp 
We continue in this paper our investigations [Mo1,Mo2] on properties of translation-invariant pure states of $\clb$ with additional symmetries that arises naturally for ground states of Heisenberg-type ferromagnetic or anti-ferromagnetic models in one lattice dimensional quantum spin chains $\clb$ and settle partially a general mathematical conjecture raised in [Ma3] on split property and its relation with asymptotic behaviour of two-point spatial correlation functions. Our main result in particular,  also settles partially some problems raised in 
[KMSW] on infinitely entangled states. Taku Matsui, in his paper [Ma3], conjectured that the exponential decaying property of two-point spatial correlation functions of a translation-invariant factor state $\omega$ will imply split property of the state $\omega$. We will prove this conjecture for a translation-invariant pure state $\omega$ on, one lattice dimensional quantum spin chain namely, $\clb=\otimes_{j \in \IZ} \!M^{(j)}_d(\IC)$, under some additional symmetries on $\omega$ which are described below and main result is stated 
as a theorem (Theorem 1.3). 

\vsp 
If $\omega$ is a translation-invariant pure state on $\clb$ then, Theorem 1.1 in [Mo1] says that $\pi_{\omega}(\clb_R)''$ is either a type-I or a type-III factor. On the other hand, if $\omega_R$ is a type-I factor state for a translation-invariant state $\omega$ then, $\omega$ is pure on $\clb$ by Theorem 2.8 in [Mo1] (for a different proof,  see [Ma3]). There exist examples [AMa] of translation-invariant pure states $\omega$ on $\clb$ for which $\omega_R$ is a type-III factor state on $\clb_R$. It is much easier [BJKW] to construct a translation-invariant state $\omega$ on $\clb$ with $\omega_R$ as type-I factor state using Popescu's dilation theory [Po]. A purely mathematical question that arises now: how to assert which type of factor states $\omega_R (\omega_L)$ are,  i.e. type-I or type-III factors, by studying additional symmetry of the state $\omega$ or asymptotic behaviour of the group of automorphisms $(\clb,\theta^n,\omega)$? 

\vsp
Let $Q \raro \tilde{Q}$ be the automorphism on $\clb$ that maps an element 
$$Q=Q_{-l}^{(-l)} \otimes Q_{-l+1}^{(-l+1)} \otimes ... \otimes Q_{-1}^{(-1)} \otimes Q_0^{(0)} \otimes Q_1^{(1)} ... \otimes Q_n^{(n)}$$ by reflecting around the point ${1 \over 2}$ of the lattice $\IZ$ to 
$$\tilde{Q}= Q_n^{(-n+1)}... \otimes Q_1^{(0)} \otimes Q_0^{(1)} \otimes Q_{-1}^{(2)} \otimes ... Q_{-l+1}^{(l)} \otimes Q_{-l}^{(l+1)}$$
for all $n,l \ge 1$ and $Q_{-l},..Q_{-1},Q_0,Q_1,..,Q_n \in 
M_d(\IC)$. 

\vsp 
For a state $\omega$ of $\clb$ we set a state $\tilde{\omega}$ of $\clb$ by 
\be 
\tilde{\omega}(Q)= \omega(\tilde{Q})
\ee
for all $Q \in \clb$. Thus $\omega \raro \tilde{\omega}$ is an affine one to one onto 
map on the convex set of states of $\clb$. The state $\tilde{\omega}$ is translation-invariant, ergodic, 
factor state,  if and only if $\omega$ is translation-invariant, ergodic, factor state respectively. We say 
a state $\omega$ is {\it lattice reflection-symmetric} or in short {\it lattice symmetric } if $\omega=\tilde{\omega}$.   
 
\vsp 
If $Q= Q^{(l)}_0 \otimes Q^{(l+1)}_1 \otimes ....\otimes Q^{(l+m)}_m$ we set
$Q^t_e={Q^t}^{(l)}_0 \otimes {Q^t}^{(l+1)}_1 \otimes ..\otimes {Q^t}^{(l+m)}_m,$
where $Q_0,Q_1,...,Q_m$ are arbitrary elements in $M_d(\IC)$ and $Q_0^t,Q^t_1,..$ stands for transpose
with respect to an orthonormal basis $(e_i)$ for $\IC^d$ (not complex conjugate) of $Q_0,Q_1,..$ 
respectively. Note that in order to indicate $Q^t_e$ depends on the basis $e$, we have used suffix $e$ which we will omit assuming that it won't confuse an attentive reader once a fixed orthonormal basis $(e_i)$ is under consideration. We define $Q^t$ by extending linearly for any $Q \in \clb_{loc}$. For a state $\omega$ of $\clb$, we define a state $\bar{\omega}$ on $\clb$ by the following prescription
\be
\bar{\omega}(Q) = \omega(Q^t)
\ee
Thus the state $\bar{\omega}$ is translation-invariant, ergodic, factor state, if and only if $\omega$ is translation-invariant, ergodic, factor state respectively. We say $\omega$ is {\it real }, if $\bar{\omega}=\omega$.

\vsp 
Unitary matrices $v \in U_d(\IC)$ act naturally on $\clb$ as group of automorphisms defined by 
\be 
\beta_{v}(Q)=(..\otimes v \otimes v \otimes ...)Q(...\otimes v^* \otimes v^* \otimes v^*...)
\ee

\vsp 
We set a conjugate linear map $Q \raro \overline{Q}$ on $\clb$ with respect to a basis $(e_i)$ for $\IC^d$ defined by extending identity action on elements  
$$..I_d \otimes |e_{i_0}\rangle \langle e_{j_0}|^{(k)} \otimes |e_{i_1}\rangle \langle e_{j_1}|^{(k+1)} \otimes 
|e_{i_n}\rangle \langle e_{j_n}|^{(k+n)} \otimes I_d ..,\;1 \le i_k,j_k \le d,\;\;k \in \IZ,\;n \ge 0$$ 
anti-linearly. Thus we have $Q^*=\overline{Q^t}$.

\vsp 
Following a well known notion [FILS], a state $\omega$ on $\clb$ is called {\it reflection positive with a twist $g_0 \in U_d(\IC)$}, if 
\be 
\omega(\clj_{g_0}(Q) Q) \ge 0
\ee 
for all $Q \in \clb_R$,  where $v_0^2=I_d$ and $\clj_{g_0}(Q) = \overline{\beta_{g_0}(\tilde{Q})}$.
 
\vsp 
Let $G$ be a compact group and $g \raro v(g)$ be a $d-$dimensional unitary representation of $G$. By $\gamma_g$ we 
denote the product action of $G$ on the infinite tensor product $\clb$ induced by $v(g)$,
\be 
\gamma_g(Q)=(..\otimes v(g) \otimes v(g)\otimes v(g)...)Q(...\otimes v(g)^*\otimes v(g)^*\otimes v(g)^*...)
\ee
for any $Q \in \clb$, i.e. $\gamma_g=\beta_{v(g)}$. We say $\omega$ is $G$-invariant, if $\omega(\gamma_g(Q))=\omega(Q)$ for all $Q \in \clb_{loc}$. If $G=U_d(\IC)$ and $v:U_d(\IC) \raro U_d(\IC)$ be the natural representation $v(g)=g$, we often use notation $\beta_g$ for $\gamma_g$ for simplification. 

\vsp 
Now we state our main theorem proved in this paper. 

\vsp
\begin{thm} 
Let $\omega$ be a pure lattice reflection-symmetric translation-invariant real ( with respect to a basis $(e_i)$ of $\IC^d$ ) state of $\clb$. If $\omega$ is also reflection positive with 
a twist $g_0 \in U_d(\IC),\;g_0^2=I_d$ then, two-point spatial correlation function of $\omega$ decays exponentially, if and only if $\omega$ is a split state i.e. $\pi_{\omega}(\clb_R)''$ is a type-I factor.
\end{thm}

\vsp 
Theorem 1.3, in particular,  also gives a sharper estimate for two-point correlation function which in particular,  makes it worth in the context of non-commutative version of {\it central limit theorem} for stationary states proved in [GV1], [GV2], [GV3] and [Ma5] in various degrees of generalities. In particular,  Theorem 1.2 in [Ma5] now says that the central limit theorem holds [Ma5] for a pure lattice reflection-symmetric translation-invariant real ( with respect to a basis $(e_i)$ of $\IC^d$ ) state $\omega$ of $\clb$, if $\omega$ is split.

\vsp 
The paper is organized as follows: In section 2, we recall basic results on Cuntz algebra $\clo_d$ and canonical amalgamated representation $\pi$ of $\tilde{\clo}_d \otimes \clo_d$ associated with a state $\psi$ of $\clo_d$ which extend GNS representation of $\clb \equiv \tilde{\mbox{UHF}}_d \otimes \mbox{UHF}_d$ associated with the state $\omega$. In section 3 we investigate the amalgamated representation $\pi$ with additional symmetries, in particular, when $\omega$ is pure. In section 4 we explore the representation $\pi$ to prove Theorem 1.3. The last section briefly reviews some well known spin chain models and illustrates implication of our main result.

\section{Mathematical Preliminaries} 

\vsp 
For the last few decades, a translation-invariant state of $\clb$ had been studied extensively in the mathematical literature, either in the framework of quantum Markov states [Ac], [FNW1], [FNW2], [FNW3] or in the frame work of representation theory of $C^*$ algebras  [Pow], [Cu], [BJ], [BJP] and [BJKW]. Our investigation in [Mo1] and [Mo2] had clubbed these two frameworks into an unified Kolmogorov's dilation theory, where inductive limit states [Sa[ are visualized in the frame work of
Kolmogorov's consistency theorem for stationary Markov processes. In this section, however, we give the basic ideas that are involved in the proof of Theorem 1.3 after recalling some known results from [BJKW], [Mo1] and [Mo2] for our present purpose.  

\vsp
First we recall that the Cuntz algebra $\clo_d ( d \in \{2,3,.., \} )$ [Cun] is the universal unital $C^*$-algebra generated by the elements $\{s_1,s_2,...,
s_d \}$ subject to the following relations:
\be 
s_i^*s_j = \delta^i_j I,\;\;\sum_{1 \le i \le d } s_is^*_i=I
\ee

\vsp
Let $\IZ_d=\{1,2,3,...,d\}$ be a set of $d$ elements. $\cli$ be the set of finite sequences
$I=(i_1,i_2,...,i_m)$ of elements, where $i_k \in \IZ_d$ and $m \ge 1$ and we use notation 
$|I|$ for the cardinality of $I$. We also include null set denoted by $\emptyset$ in the collection $\cli$ and set $s_{\emptyset }=s^*_{\emptyset}=I$ identity of $\clo_d$ and $s_{I}=s_{i_1}......s_{i_m} \in \clo_d $ and $s^*_{I}=s^*_{i_m}...s^*_{i_1} \in \clo_d$. 

\vsp
The group $U_d(\IC)$ of $d \times d$ unitary matrices acts canonically on $\clo_d$ as follows:
$$\beta_g(s_i)=\sum_{1 \le j \le d}\overline{g^j_i}s_j$$
for $g=((g^i_j) \in U_d(\IC)$. In particular,  the gauge action is defined by
$$\beta_z(s_i)=zs_i,\;\;z \in \IT = S^1= \{z \in \IC: |z|=1 \}.$$
The fixed point sub-algebra of $\clo_d$ under the gauge action i.e., 
$\{x \in \clo_d: \beta_z(x)=x,\;z \in S^1 \}$ is the closure of 
the linear span of all Wick ordered monomials of the form
$$s_{i_1}...s_{i_k}s^*_{j_k}...s^*_{j_1}:\;I=(i_1,..,i_k),J=(j_1,j_2,..,j_k)$$
and is isomorphic to the uniformly hyper-finite $C^*$ subalgebra
$$\clb_R =\otimes_{1 \le k < \infty}\!M^{(k)}_d(\IC)$$
of $\clb$, where the isomorphism carries the Wick ordered monomial above 
into the matrix element
$$|e^{i_1}\rangle \langle e_{j_1}|^{(1)} \otimes |e^{i_2}\rangle\langle e_{j_2}|^{(2)} \otimes....\otimes |e^{i_k}\rangle\langle e_{j_k}|^{(k)} \otimes 1 \otimes 1 ....$$
We use notation $\mbox{UHF}_d$ for the fix point $C^*$ sub-algebra of $\clo_d$ under the gauge group action 
$(\beta_z:z \in S^1)$. The restriction of $\beta_g$ to $\mbox{UHF}_d$ is then carried into action
$$Ad(g)\otimes Ad(g) \otimes Ad(g) \otimes ....$$
on $\clb_R$.

\vsp
We also define the canonical endomorphism $\lambda$ on $\clo_d$ by
\be 
\lambda(x)=\sum_{1 \le i \le d}s_ixs^*_i
\ee
and the isomorphism carries $\lambda$ restricted to $\mbox{UHF}_d$ into the one-sided shift
$$y_1 \otimes y_2 \otimes ... \raro 1 \otimes y_1 \otimes y_2 ....$$
on $\clb_R$. We note for all $g \in U_d(\IC)$ that $\lambda \beta_g = \beta_g \lambda$ 
on $\clo_d$ and so in particular,  also on $\mbox{UHF}_d$. 

\vsp 
A family $(v_k:1 \le k \le d)$ of contractive operators on a Hilbert space $\clk$ is called {\it Popescu's elements} [Po], if 
\be 
\sum_k v_kv_k^*=I_{\clk}
\ee
For a family of Popescu's elements $(v_k:1 \le k \le d)$ on a Hilbert space $\clk$, we define a unital completely positive map $\tau$ on $\clb(\clk)$ by 
\be 
\tau(x)=\sum_k v_k x v_k^*,\;x \in \clb(\clk)
\ee
and $\tau$-invariant elements $\clb_{\tau}(\clk)$ in $\clb(\clk)$ by 
\be 
\clb(\clk)_{\tau}=\{ x \in \clb(\clk): \tau(x)=x \}
\ee 
We also note that the group action $(\beta_g)$ of $U_d(\IC)$ on the collection of Popescu's elements $(v_i)$ defined by 
$$\beta_g(v_i)= \sum_{1 \le j \le d}\overline{g^j_i}v_j,\;\;1 \le j \le d$$
keeps $\clb(\clk)_{\tau}$ unperturbed. 

\vsp 
We recall Proposition 2.4 in [Mo2] with little more details in the following proposition without proof.  

\vsp 
\begin{pro} 
Let $(\clh,\pi,\Omega)$ be the GNS representation of a $\lambda$ invariant state 
$\psi$ on $\clo_d$ and $P$ be the support projection of the normal state $\psi_{\Omega}(X)=\langle\Omega,X\Omega\rangle$ in the 
von-Neumann algebra $\pi(\clo_d)''$. Then the following holds:

\vsp
\NI (a) $P$ is a sub-harmonic projection for the endomorphism $\Lambda(X)=\sum_k S_kXS^*_k$ on $\pi(\clo_d)''$
i.e. $\Lambda(P) \ge P$ satisfying the following:

\NI (i) $\Lambda_n(P) \uparrow I$ as $n \uparrow \infty$;

\NI (ii) $PS^*_kP=S^*_kP,\;\;1 \le k \le d$;

\NI (iii) $\sum_{1 \le k \le d} v_kv_k^*=I;$ 

where $S_k=\pi(s_k)$ and $v_k=PS_kP$ for $1 \le k \le d$;

\vsp
\NI (b) For any $I=(i_1,i_2,...,i_k),J=(j_1,j_2,...,j_l)$ with $|I|,|J| < \infty$ we have $\psi(s_Is^*_J) =
\langle\Omega,v_Iv^*_J\Omega\rangle$ and the vectors $\{ S_If: f \in \clk,\;|I| < \infty \}$ are total in $\clh$;

\vsp
\NI (c) The von-Neumann algebra $\clm=P\pi(\clo_d)''P$, acting on the Hilbert space
$\clk$ i.e. range of $P$, is generated by $\{v_k,v^*_k:1 \le k \le d \}''$ and the normal state
$\phi(x)=\langle\Omega,x \Omega\rangle$ is faithful on the von-Neumann algebra $\clm$.

\NI (d) The self-adjoint part of the commutant of $\pi(\clo_d)'$ is norm and order isomorphic to the space
of self-adjoint fixed points of the completely positive map $\tau$. The isomorphism takes $X' \in \pi(\clo_d)'$
onto $PX'P \in \clb_{\tau}(\clk)$. Furthermore, 
$\clm' = \clb_{\tau}(\clk)$.

\vsp
\NI (e) The following statements are equivalent:

\NI (i) $\psi$ is a factor state;

\NI (i) $\clm$ is a factor.

\vsp
Conversely, let $v_1,v_2,...,v_d$ be a family of bounded operators on a Hilbert space
$\clk$ so that $\sum_{1 \le k \le d} v_kv_k^*=I$. Then there exists a unique up to
unitary isomorphism Hilbert space $\clh$, a projection operator $P$ on $\clh$ with range equal to $\clk$ 
and a family of isometries
$\{S_k:,\;1 \le k \le d \}$ satisfying Cuntz's relation (10) so that
\be
PS^{*}_{k}P=S_k^*P=v^*_k
\ee
for all $1 \le k \le d$ and $\clk$ is cyclic for the representation i.e. the vectors
$\{ S_I\clk: |I| < \infty \}$ are total in $\clh$.

Moreover the following holds:

\NI (i) $\Lambda_n(P) \uparrow I$ as $n \uparrow \infty$;

\NI (ii) For any $D \in \clb_{\tau}(\clk)$, $\Lambda_n(D) \raro X'$ weakly as $n \raro \infty$
for some $X'$ in the commutant $\{S_k,S^*_k: 1 \le k \le d \}'$ so that $PX'P=D$. Moreover
the self adjoint elements in the commutant $\{S_k,S^*_k: 1 \le k \le d \}'$ is isometrically
order isomorphic with the self adjoint elements in $\clb_{\tau}(\clk)$ via the
surjective map $X' \raro PX'P$. 

\NI (iii) $\{v_k,v^*_k,\;1 \le k \le d \}' \subseteq \clb_{\tau}(\clk)$ and equality holds,  if and only if
$P \in \{S_k,S_k,\;1 \le k \le d \}''$.

\vsp
\NI (iv) Let $\clm$ be a von-Neumann algebra generated by the family $\{v_k: 1 \le k \le d \}$ of operators on Hilbert space 
$\clk$ and $\clm'=\clb_{\tau}(\clk)$. Then for any  $\tau$-invariant faithful normal state $\phi$ on $\clm$ there exists a 
$\lambda$-invariant state $\psi$ on $\clo_d$ defined by
$$\psi(s_Is^*_J)=\phi(v_Iv^*_J),\;|I|,|J| < \infty $$
so that its GNS space associated with $(\clm,\phi)$ is identified with the support projection of $\psi$ in $\pi(\clo_d)''$, where $(\clh,\pi,\Omega)$ is the GNS space of $(\clo_d,\psi)$. 

\vsp
Furthermore,  for a given $\lambda$-invariant state $\psi$, the family $(\clk,\clm,v_k\;1 \le k \le d,\phi)$ satisfying (iv) 
is determined uniquely up to unitary conjugation. 

\end{pro} 

\vsp 
Our next two propositions are adapted from results in section 6 and section 7 of [BJKW] as stated in the present form in Proposition 2.5 and Proposition 2.6 in [Mo2]. 

\vsp 
\begin{pro} 
Let $\psi$ be a $\lambda$ invariant factor state on $\clo_d$ and $(\clh,\pi,\Omega)$
be its GNS representation. Then the following holds:

\NI (a) The closed subgroup $H=\{z \in S^1: \psi \beta_z =\psi \}$ is equal to 

$$\{z \in S^1: \beta_z \mbox{extends to an automorphism of } \pi(\clo_d)'' \} $$ 

\NI (b) Let $\clo_d^{H}$ be the fixed point sub-algebra in $\clo_d$ under the gauge group $\{ \beta_z: z \in H \}$. Then  
$\pi(\clo_d^{H})'' = \pi(\mbox{UHF}_d)''$.

\NI (c) If $H$ is a finite cyclic group of $k$ many elements and $\pi(\mbox{UHF}_d)''$ is a factor, 
then $\pi(\clo_d)'' \bigcap \pi(\mbox{UHF}_d)' \equiv \IC^m$,  where $1 \le m \le k$.  
\end{pro}

\vsp 
Let $\psi$ be a $\lambda$-invariant factor state of $\clo_d$ as in Proposition 2.2. Let 
$z \raro U_z$ be the unitary representation of $H$ in the GNS 
space $(\clh,\pi,\Omega)$ associated with the state $\psi$ of 
$\clo_d$ defined by 
\be 
U_z\pi(x)\Omega=\pi(\beta_z(x))\Omega
\ee
so that $\pi(\beta_z(x))=U_z\pi(x)U_z^*$ for $x \in \clo_d$. We use same notation $(\beta_z:z \in H)$ for its normal extension as automorphism on $\pi(\clo_d)''$. 
Furthermore,  $\langle \Omega,P\beta_z(I-P)P \Omega \rangle=0$ since $\psi=\psi \beta_z$ for 
$z \in H$. $P$ being the support projection of $\psi$ in $\pi(\clo_d)''$, we have $P\beta_z(I-P)P=0$ i.e. $\beta_z(P) \ge P$ for all $z \in H$. Since $H$ is group, 
we conclude that $\beta_z(P)=P$ i.e. $PU_z=U_zP$ for all $z \in H$. Thus by Proposition 
2.2 (b), $P \in \pi(\mbox{UHF}_d)''$. We define von-Neumann subalgebra $\clm_0$ of $\clm$ by 
\be 
\clm_0=P\pi(\mbox{UHF}_d)''P
\ee
i.e. $\clm_0$ is weak$^*$ closure of vector space $\{v_Iv_J^*:|I|=|J|\}$ by Proposition 2.1 (a). Let $\clk_0$ be the Hilbert subspace of $\clk$ equal to the range of $[\clm_0\Omega]$. 
Then $\clm_0$ can be realized as a von-Neumann subalgebra of $\clb(\clk_0)$, however 
its commutant in $\clb(\clk_0)$ could be different from commutant $\clm_0'$ taken in $\clb(\clk)$. 

\vsp 
Since endomorphism $\Lambda(X)=\sum_k \pi(s_k)X\pi(s_k^*)$ preserves $\pi(\mbox{UHF}_d)''$
$\tau$ also preserves $\clm_0$. Let $\phi_0$ be the restriction of $\phi$ to $\clm_0$. Thus $(\clm_0,\tau^n,\;n \ge 1, \phi_0)$ is a quantum dynamical system [Mo1] of a completely positive map $\tau$ on $\clm_0$ with a faithful normal invariant state $\phi_0$.    

\vsp 
We consider now the group of automorphism $(\beta_z:z \in H)$ on $\clb(\clk)$ defined by
$\beta_z(a)=u_zau_z^*$, where $z \raro u_z=PU_zP$ is the unitary representation 
of $H$ in $\clk$. Let 
\be
u_z=\sum_{k \in \hat{H}} z^k P_k
\ee 
be 
Stone-Naimark-Ambrose-Godement (SNAG) decomposition [Mac49] of the unitary representation 
$z \raro u_z$ into its dual group $\hat{H}$. So we have 
$$P_k=[v_Iv_J^*\Omega: |I|-|J|=k]$$ for $k \in \hat{H}$ and in particular,  
$$P_0=[\clm_0\Omega]$$
Furthermore,  a bounded linear or anti-linear operator $a$ on $\clk$ keeps each subspace 
$P_k$ invariant for $k \in \hat{H}$,  if and only if $\beta_z(a)=u_zau_z^*=a$ for all $z \in H$.  

\vsp
Let $\omega'$ be a $\lambda$-invariant state on the $\mbox{UHF}_d$ sub-algebra of $\clo_d$. Following [BJKW, section 7] and $\omega$ be the inductive limit state $\omega$ of $\clb \equiv \tilde{\mbox{UHF}}_d \otimes \mbox{UHF}_d$. In other words $\omega'=\omega_R$ once we make the identification $\mbox{UHF}_d$ with $\clb_R$.
We consider the set 
$$K_{\omega}= \{ \psi: \psi \mbox{ is a state on } \clo_d \mbox{ such that } \psi \lambda =
\psi \mbox{ and } \psi_{|\mbox{UHF}_d} = \omega_R \}$$
By taking invariant mean on an extension of $\omega_R$ to $\clo_d$, we verify that $K_{\omega}$ is non empty and 
$K_{\omega}$ is clearly convex and compact in the weak topology. In case $\omega$ is an ergodic state ( extremal state ) then, $\omega_R$ is as well an extremal state in the set of $\lambda$-invariant states of $\clb$. Thus
$K_{\omega}$ is a face in the $\lambda$ invariant states. Now we recall Lemma 7.4 
of [BJKW] in the following proposition which quantifies what we can gain 
by considering a factor state on $\clo_d$ instead of its restriction to 
$\mbox{UHF}_d$.

\vsp 
\begin{pro} Let $\omega$ be an ergodic state of $\clb$. Then $\psi \in K_{\omega}$ is an extremal point in
$K_{\omega}$,  if and only if $\psi$ is a factor state. Moreover any other extremal point in $K_{\omega}$
is of the form $\psi \beta_z$ for some $z \in S^1$ and 
$H=\{z \in S^1: \psi \beta_z =\psi \}$ is independent of the extremal point $\psi \in K_{\omega}$.   
\end{pro} 

\vsp 
In Proposition 2.1 (b) we have taken an arbitrary element $\psi \in K_{\omega}$ to find 
a family of Popescu's elements $\clp=(\clk,v_i \in \clm,1 \le i \le d,\;\Omega)$ in its 
support projection and arrived at a representation of $\omega$ given by
\be 
\omega(|e^{i_1}\rangle \langle e_{j_1}|^{(1)} \otimes |e^{i_2}\rangle\langle e_{j_2}|^{(2)} \otimes....\otimes |e^{i_k}\rangle\langle e_{j_k}|^{(k)} \otimes 1 \otimes 1 ..)=\phi(v_Iv_J^*),
\ee
where $I=(i_1,i_2,..,i_k)$ and $J=(j_1,j_2,..,j_k)$. However,  such a representation need 
not be unique even upto unitary conjugation unless $K_{\omega}$ is a singleton set. Nevertheless by Proposition 2.3 for a factor state $\omega$, two extreme points $\psi$ and $\psi'$ in $K_{\omega}$ being related by $\psi'=\psi \beta_z$ for some $z \in S^1$, Popescu's elements $\clp=\{\clk,v_k:1 \le k \le d,\;\sum_kv_kv_k^*=I_{\clk} \}$ 
and $\clp'=\{\clk',v'_k:1 \le k \le d,\; \sum_k v'_k(v')^*_k=I_{\clk'} \}$ associated with support projections of $\psi$ and $\psi'$ in $\pi_{\psi}(\clo_d)''$ and 
$\pi_{\psi'}(\clo_d)''$ respectively are unitary equivalent modulo a gauge modification i.e. by Proposition 2.1 there exists a unitary operator $u:\clk \raro \clk'$ and $z \in S^1$ so that $uv_ku^*=z v'_k$ for all $1 \le k \le d$. In other words we find a one-one correspondence between 
\be
\omega \Leftrightarrow \omega_R \Leftrightarrow K^{ext}_{\omega} \Leftrightarrow \clp_{ext} \Leftrightarrow (\clm,\tau,\phi) 
\ee  
modulo unitary conjugations and phase factors, where $K^{ext}_{\omega}$ is the set of extreme points in $K_{\omega}$ and $\clp_{ext}$ is the set of Popescu's elements associated with extreme points of $K_{\omega}$ on their support projections 
of the states given as in Proposition 2.1. 
Furthermore,  for the support projection $P=[\pi(\clo_d)'\Omega]$, we have $\beta_z(P)=P$ for all $z \in H$. Hence by Proposition 2.2 (b), we have $P \in \pi(\mbox{UHF}_d)''$. So $\clm_0=P\pi(\clo_d)''P$ is a von-Neumann algebra in its own right and $\clm_0 \subseteq \clm$ and $\tau$ takes elements of $\clm_0$ to itself. Thus $(\clm_0,\tau,\phi_0)$ is a semi-group of unital completely positive maps with a faithful normal invariant state $\phi_0$, the restriction of $\phi$ to $\clm_0$ 
and such a triplet $(\clm_0,\tau,\phi_0)$ is canonically associated with the state $\omega$ modulo unitary conjugation. Thus it is natural to expect that various properties of $\omega$ are related to that of $(\clm_0,\tau,\phi)$. We have already explored 
purity of $\omega$ in [Mo1] to find its precise relation with the asymptotic behaviour of the dynamics $(\clm_0,\tau^n,\phi_0)$ as $n \raro \infty$. Along with $(\clm_0,\tau^n,\phi_0)$, asymptotic behaviour of the {\it dual dynamics } $(\clm'_0,\tilde{\tau}^n,\phi_0)$ ( defined by D Petz [OP] following a work of Accardi-Cecchini [AC] ) also played an important role in our analysis. We now recall the details of it and explain how it is related to symmetry (5) of $\omega$.

\vsp
Since $\phi$ is a faithful state, $\Omega \in \clk$ is a cyclic and separating vector for $\clm$ and the closure of the closable operator $S_0:a\Omega \raro a^*\Omega,\;a \in \clm, S$ possesses a polar decomposition $S=\clj \Delta^{1/2}$, where $\clj$ is an anti-unitary and $\Delta$ is a non-negative self-adjoint operator on $\clk$. Tomita's [BR] theorem says that 
$\Delta^{it} \clm \Delta^{-it}=\clm,\;t \in \IR$ and $\clj \clm \clj=\clm'$, where $\clm'$ is the
commutant of $\clm$. We define the modular automorphism group
$\sigma=(\sigma_t,\;t \in \IT )$ on $\clm$
by
$$\sigma_t(a)=\Delta^{it}a\Delta^{-it}$$ which satisfies the modular relation
$$\phi(a\sigma_{-{i \over 2}}(b))=\phi(\sigma_{{i \over 2}}(b)a)$$
for any two analytic elements $a,b$ for the group of automorphisms $(\sigma_t)$. A more useful modular relation used frequently in this paper is given by 
\be 
\phi(\sigma_{-{i \over 2}}(a^*)^* \sigma_{-{i \over 2}}(b^*))=\phi(b^*a)
\ee 
which shows that $\clj a\Omega= \sigma_{-{i \over 2}}(a^*)\Omega$ for an analytic element $a$ for the automorphism group $(\sigma_t)$. Anti unitary operator $\clj$ and the group of automorphism $\sigma=(\sigma_t,\;t \in \IR)$ are called {\it Tomita's conjugate operator} and {\it modular automorphisms } associated with $\phi$ respectively. 

\vsp 
The state $\phi(a)= \langle \Omega,x \Omega \rangle $ on $\clm$ being faithful and invariant of $\tau:\clm \raro \clm$, we find a unique unital completely positive map 
$\tilde{\tau}:\clm' \raro \clm'$ ([section 8 in [OP] ) satisfying the duality relation 
\be 
\langle b\Omega,\tau(a)\Omega \rangle =  \langle \tilde{\tau}(b)\Omega,a\Omega \rangle 
\ee
for all $a \in \clm$ and $b \in \clm'$. For a proof,  we refer 
to section 8 in the monograph [OP] or section 2 in [Mo1]. 

\vsp 
Since $\tau(a)=\sum_{1 \le k \le d} v_kav_k^*,\;x \in \clm$ is an {\it inner map } i.e. each $v_k \in \clm$, we have an explicit formula for $\tilde{\tau}$ as follows: For 
each $1 \le k \le d$, we set contractive operator 
\be 
\tilde{v}_k = \overline{ \clj \sigma_{i \over 2}(v^*_k) \clj } \in \clm'
\ee 
That $\tilde{v}_k$ is indeed well defined as an element in $\clm'$ given in section 8 in [BJKW]. By the modular relation (17), we have  
\be  
\sum_k \tilde{v}_k \tilde{v}_k^*=I_{\clk}\;\;\mbox{and}\;\;
\tilde{\tau}(b)=\sum_k \tilde{v}_kb\tilde{v}^*_k,\; b \in \clm' 
\ee
Moreover, if $\tilde{I}=(i_n,..,i_2,i_1)$ for $I=(i_1,i_2,...,i_n)$, we have 
$$\tilde{v}^*_I\Omega$$
$$=\clj \sigma_{i \over 2}(v_{\tilde{I}})^*\clj\Omega$$
$$= \clj \Delta^{1 \over 2}v_{\tilde{I}}\Omega$$
$$=v^*_{\tilde{I}}\Omega$$
and    
\be
\phi(v_Iv^*_J)= \phi(\tilde{v}_{\tilde{I}}\tilde{v}^*_{\tilde{J}}),\; |I|,|J| < \infty 
\ee
We also set $\tilde{\clm}$ to be the von-Neumann algebra generated by $\{\tilde{v}_k: 1 \le k \le d \}$. 
Thus $\tilde{\clm} \subseteq \clm'$.

\vsp 
Since $S_0\beta_z(a)\Omega=\beta_z(a^*)\Omega$ for all $a \in \clm$, we have $S_0u_z=u_zS_0$ on $\clm\Omega$. Once again by uniqueness of polar decomposition for $S= \clj \Delta^{1 \over 2}$, we get $u_z \clj u_z^*=\clj$ and $u_z \Delta^{1 \over 2} u_z^*=\Delta^{1 \over 2}$. 
In particular,  $\clj,\Delta^{1 \over 2}$ commutes with $P_k$ for all $k \in \hat{H}$,  where 
$P_k$ are projections defined in (18). Furthermore,  since $E=\int_{z \in H}\beta_z dz$ is a norm one projection ( i.e. a unital completely positive map $E:\clm \raro \clm_0$ satisfying the bi-module property, i.e. $E(zxy)=zE(x)y,\;x \in \clm, z,y \in \clm_0$ ) from $\clm$ to the fixed point von-Neumann sub-algebra $\clm_0$ of $\clm$, the modular group of automorphisms $(\sigma_t)$ keep $\clm_0$ invariant i.e. $\sigma_t(\clm_0)=\clm_0$ 
for all $t \in \IR$ by a Theorem of M. Takesaki [Ta]. Note also that $\clm_0=P\pi(\mbox{UHF}_d)''P$ as a von-Neumann algebra with its cyclic space $\clk_0=[\clm_0 \Omega]$. Thus,  we have von-Neumann algebra $\clm_0$ acting on $\clk_0$ and the map $x \raro \tau(x)$, once restricted to $\clm_0$ admits an adjoint completely positive map which is the restriction of $\tilde{\tau}$ on $\clm'_0$ i.e. 
$$\tilde{\tau}(y)= \sum_{1 \le k \le d} \tilde{v}_ky\tilde{v}_k^*$$ 
for all $y \in \clm_0'$ satisfying the restricted duality relation of (22) 
given below:
$$\langle y\Omega,\tau(x)\Omega \rangle=\langle \tilde{\tau}(y)\Omega,x\Omega\rangle$$
for all $x \in \clm_0,y \in \clm_0'$, where $\clm_0'$ is the commutant of $\clm_0$ in $\clb(\clk_0)$.  

\vsp 
We denote by $\clk^{1 \over 2}$ the Hilbert space completion of the vector space $\clm$ with inner product given by 
$$\langle \langle x,y \rangle \rangle = \langle x \clj y \clj \rangle$$
Then $T:\clk^{1 \over 2} \raro \clk^{1 \over 2}$ is a contractive operator defined by extending the map
\be 
Tx \raro \tau(x),\;x \in \clm
\ee 
and it restriction to $\clm_0$ has a contractive extension 
to its Hilbert space closure denoted by $\clk^{1 \over 2}_0$.   

\vsp 
A non-trivial symmetry of $\omega$ will determine a unique affine map on $K_{\omega}$ and thus taking an extremal element of $K_{\omega}$ to another extremal element of $K_{\omega}$. Since associated Poposecu's elements on support projections of an extremal element are determined uniquely modulo a unitary conjugation, each symmetry will give rises to an undetermined unitary operators intertwining Popescu's elements modulo a gauge group action. Basic strategy here is to find algebraic relation between Cuntz state $\psi$ and associated Popescu's elements $(\clk,v_k:1 \le k \le d)$ in its support projection with 
its dual Cuntz state $\tilde{\psi}$ associated with dual Popescu's elements $(\tilde{v}_k,1 \le k \le d)$. Using symmetries of $\omega$ given in (5) and (6), additional condition of refection positivity (8) ensures that $T:\clk_0^{1 \over 2} \raro \clk_0^{1 \over 2}$ is self-adjoint. Thus the spectrum of $T$ commands on the behaviour of two-point spatial correlations function and also on split property of $\omega$. While studying symmetry (5) of $\omega$, we need equality of the support projections of $\psi$ and $\tilde{\psi}$ in order to find algebraic relations between their Popescu's elements. To that end we recall results from [Mo2] in the next paragraph.      

\vsp 
Let $\tilde{\clo}_d$ be an another copy of Cuntz algebra $\clo_d$ and $\pi$ be the Popescu's prescription [Po] (Theorem 5.1 in [BJKW] or Proposition 2.1 in [Mo2]) 
of a minimal Stinespring representation $\pi: \tilde{\clo}_d \raro \clb(\tilde{\clh})$ associated with the completely positive map $\tilde{s}_I\tilde{s}_J^* \raro \tilde{v}_I\tilde{v}_J^*,\;|I|,|J| < \infty $ so that
$$P\pi(\tilde{s}^*_i)P=\pi(\tilde{s}_i^*)P=\tilde{v}^*_i$$ 
Furthermore,  we have 
a {\it dual} state $\tilde{\psi}$ of $\tilde{\clo}_d$ defined by 
$$\tilde{\psi}(\tilde{s}_I \tilde{s}^*_J)=\langle \Omega,\tilde{S}_I\tilde{S}^*_J \Omega \rangle$$
\be 
=\phi(\tilde{v}_I\tilde{v}^*_J)
\ee 
However, by the converse part of Proposition 2.1, $P:\tilde{\clh} \raro \clk$ is also the support projection 
of $\tilde{\psi}$ in $\pi(\tilde{\clo}_d)''$, if and only if 
$$ 
\{y \in \clb(\clk): \sum_k \tilde{v}_ky\tilde{v}_k^*= y \} = \tilde{\clm}'
$$

\vsp 
Let $\psi$ be a $\lambda$-invariant state on $\clo_d$ and $\tilde{\psi}$ be the state on $\tilde{\clo}_d$, which is another copy of $\clo_d$, defined by 
\be 
\tilde{\psi}(\tilde{s}_I\tilde{s}_J^*)=\psi(s_{\tilde{J}}s_{\tilde{I}}^*)
\ee 
for all $|I|,|J| < \infty$ and $(\clh_{\tilde{\psi}},\pi_{\tilde{\psi}},\Omega_{\tilde{\psi}})$ be the GNS 
space associated with $(\tilde{\clo}_d,\tilde{\psi})$. That $\tilde{\psi}$ is well defined and coincide with our earlier definition of dual state $\tilde{\psi}$ given in (27) 
follows once we check by (25) that $$\psi(s_{\tilde{J}}s_{\tilde{I}}^*)=\phi(v_{\tilde{J}}v_{\tilde{I}}^*) = 
\phi(\tilde{v}_I\tilde{v}_J^*)$$
for all $|I|,|J| < \infty$. 
      
\vsp
Similarly for any translation-invariant state $\omega$ on $\clb$, we define translation-invariant state $\tilde{\omega}$ of $\clb$ by 
$$\tilde{\omega}(Q_{-l}^{(-l)} \otimes Q_{-l+1}^{(-l+1)} \otimes ... \otimes Q_{-1}^{(-1)} \otimes Q_0^{(0)} \otimes Q_1^{(1)} ... \otimes Q_n^{(n)})$$
\be 
= \omega(Q_n^{(-n+1)}... \otimes Q_1^{(0)} \otimes Q_0^{(1)} \otimes Q_{-1}^{(2)} 
\otimes ... Q_{-l+1}^{(l)} \otimes Q_{-l}^{(l+1)})
\ee
for all $n,l \ge 1$ and $Q_{-l},..Q_{-1},Q_0,Q_1,..,Q_n \in 
M_n(\IC)$,  where $Q^{(k)}$ is the matrix $Q$ at lattice point $k$ and extending 
linearly to any $Q \in \clb$. We also note that $\psi \in K_{\omega}$,  if and only if 
$\tilde{\psi} \in K_{\tilde{\omega}}$. 

\vsp 
The map $\psi \raro \tilde{\psi}$ is bijective and affine between the convex sets of $\lambda$ and $\tilde{\lambda}$ invariant states of $\clo_d$ and $\tilde{\clo}_d$ respectively. In particular,  the map $\psi \raro \tilde{\psi}$ takes an element from $K_{\omega}$ to $K_{\tilde{\omega}}$ and the map is once more bijective and affine. 
Hence for any extremal point $\psi \in K_{\omega}$, $\tilde{\psi}$ is also an extremal element in $K_{\tilde{\omega}}$. Using Power's criterion (4) we also verify here that $\omega$ is an extremal point in the convex set of translation-invariant states if 
and only if $\tilde{\omega}$ is an extremal point in the convex set of translation-invariant states. However,  such a conclusion for a pure state $\omega$ is not so obvious. 
We have the following useful proposition stated in full generality keeping in mind its future applications [Mo3]. 

\vsp
\begin{pro}
Let $\omega$ be an extremal translation-invariant state on $\clb$ and $\psi \raro \tilde{\psi}$ be the map defined from $\lambda$ invariant states on $\clo_d$ to $\tilde{\lambda}$ invariant states of $\tilde{\clo}_d$ by (28). Then the following holds:

\NI (a) $\psi \in K_{\omega}$ is a factor state,  if and only if $\tilde{\psi} \in K_{\tilde{\omega}}$ is a factor state. 

\NI (b) $\omega$ is pure, if and only if $\tilde{\omega}$ is pure. 

\NI (c) Let $(\clk,\tilde{\clm},\tilde{v}_k:1 \le k \le d)$ be the dual Popescu's 
elements defined in (23) with associated dual state $\tilde{\psi}$ of $\tilde{\clo}_d$ defined in (28). Then $P$ is the support projection of $\tilde{\psi}$ in its GNS representation, if and only if $\clb(\clk)_{\tilde{\tau}}=\tilde{\clm}'$ 
\end{pro} 

\vsp
\begin{proof} 
Since $\omega$ is an extremal translation-invariant state, by Power's criterion (4) $\tilde{\omega}$ is also an 
extremal state. Since an extremal point of $K_{\omega}$ is mapped to an extremal point in $K_{\tilde{\omega}}$ by the affine property 
of the map $\psi \raro \tilde{\psi}$, we conclude by Proposition 2.3 that $\psi$ is a factor state,  if and only if $\tilde{\psi}$
is a factor state. For (b) note that $\tilde{xy}=\tilde{x} \tilde{y}$ and $\tilde{x^*}=\tilde{x}^*$ by our definition. Thus 
$\tilde{\omega}(x^*y)= \omega(\tilde{x^*y})= \omega((\tilde{x})^*\tilde{y})$. Thus one can easily construct a unitary operator between 
the two GNS spaces associated with $(\clb,\omega)$ and $(\clb,\tilde{\omega})$ intertwining two representation modulo a reflection i.e. 
$U \pi_{\omega}(x) U^* = \pi_{\tilde{\omega}}(\tilde{x})$ and $U\Omega_{\omega} =\Omega_{\tilde{\omega}}$. Thus (b) is now obvious.  

\vsp 
The last statement (c) is the converse part of Proposition 2.1 applied to dual Popescu's elements $(\clk,\tilde{\clm},\tilde{v}_k:1 \le k \le d)$.  
\end{proof}

\vsp
Now we briefly recall the amalgamated representation $\pi$ of $\tilde{\clo}_d \otimes \clo_d$ [BJKW,Mo2] and a necessary and sufficient condition for a translation-invariant factor state $\omega$ to be pure. This criteria in particular,  proves that the equality 
$\clb(\clk)_{\tilde{\tau}}=\tilde{\clm}'$ in Proposition 2.4 (c) holds,  if and only if $\omega$ is pure.   

\vsp
Following [BJKW] we consider the amalgamated tensor product $\clh \otimes_{\clk} \tilde{\clh}$ of $\clh$ with 
$\tilde{\clh}$ over the joint subspace $\clk$. It is the completion of the quotient of the set 
$$\IC \bar{I} \otimes \IC I \otimes \clk,$$ 
where $\bar{I},I$ both consisting of all finite sequences with elements in $\{1,2, ..,d \}$, by the equivalence relation 
defined by a semi-inner product defined on the set by requiring
$$ \langle \bar{I} \otimes I \otimes f,\bar{I}\bar{J} \otimes IJ \otimes g \rangle = \langle f,\tilde{v}_{\bar{J}}v_Jg \rangle, $$
$$ \langle \bar{I}\bar{J} \otimes I \otimes f, \bar{I} \otimes IJ \otimes g \rangle  = \langle \tilde{v}_{\bar{J}}f,v_Jg \rangle $$
and all inner product that are not of these form are zero. We also define
two commuting representations $(S_i)$ and $(\tilde{S}_i)$ of $\clo_d$ on
$\clh \otimes_{\clk} \tilde{\clh}$ by the following prescription:
$$S_I\lambda(\bar{J} \otimes J \otimes f)=\lambda(\bar{J} \otimes IJ \otimes f),$$
$$\tilde{S}_{\bar{I}}\lambda(\bar{J} \otimes J \otimes f)=\lambda(\bar{J}\bar{I} \otimes J \otimes f),$$
where $\lambda$ is the quotient map from the index set to the Hilbert space. Note that the subspace generated by
$\lambda(\emptyset \otimes I \otimes \clk)$ can be identified with $\clh$ and earlier $S_I$ can be identified
with the restriction of $S_I$ defined here. Same is valid for $\tilde{S}_{\bar{I}}$. The subspace $\clk$ is
identified here with $\lambda(\emptyset \otimes \emptyset \otimes \clk)$. 
Thus $\clk$ is a cyclic subspace for the representation $$\tilde{s}_j \otimes s_i \raro \tilde{S}_j S_i$$ 
of $\tilde{\clo}_d \otimes \clo_d$ in the amalgamated Hilbert space. Let $P$ be the projection on $\clk$. Then we
have 
$$S_i^*P=PS_i^*P=v_i^*$$
$$\tilde{S}_i^*P=P\tilde{S}_i^*P=\tilde{v}^*_i$$
for all $1 \le i \le d$. 

We sum up result required in the following proposition.

\begin{pro} 
Let $\psi$ be an extremal element in $K_{\omega}$ and $(\clk,v_k,\;1 \le k \le d)$ be the Popescu's elements 
in the support projection of $\psi$ in $\pi(\clo_d)''$ described in Proposition 2.1 and $(\clk,\tilde{v}_k,\;1 \le k \le d)$ 
be the dual Popescu's elements and $\pi$ be the amalgamated representation of $\tilde{\clo}_d \otimes \clo_d$. Then the 
following holds:

\NI (a) For any $1 \le i,j \le d$ and $|I|,|J|< \infty$ and $|\bar{I}|,|\bar{J}| < \infty$
$$ \langle \Omega,\tilde{S}_{\bar{I}}\tilde{S}^*_{\bar{J}} S_iS_IS^*_JS^*_j \Omega \rangle = \langle \Omega, 
\tilde{S}_i \tilde{S}_{\bar{I}}\tilde{S}^*_{\bar{J}}\tilde{S}^*_jS_IS^*_J \Omega \rangle ;$$

\NI (b) The vector state $\psi_{\Omega}$ on $$\tilde{\mbox{UHF}}_d \otimes \mbox{UHF}_d \equiv
\otimes_{-\infty}^0 M_d \otimes_1^{\infty}M_d \equiv \otimes_{\IZ} M_d$$ is equal to $\omega$;

\NI (c) $\pi(\tilde{\clo}_d \otimes \clo_d)''= \clb(\tilde{\clh} \otimes_{\clk} \clh)$ and 
$\clm \vee \tilde{\clm} = \clb(\clk)$;

\NI (d) Let $E$ and $\tilde{E}$ be the support projection of $\psi$ in $\pi(\clo_d)''$ and 
$\pi(\tilde{\clo})d)''$ respectively i.e. $E=[\pi(\clo_d)'\Omega]$ and $\tilde{E}=[\pi(\tilde{\clo}_d)'\Omega]$. If $F=[\pi(\cld_d)''\Omega]$ and $\tilde{F}=[\pi(\tilde{\clo}_d)''\Omega]$ then, the following statements are 
equivalent:

\NI (i) $\omega$ is pure;

\NI (ii) $E=\tilde{F}$ and $\tilde{E}=F$; 

\NI (iii) $P=\tilde{E}F$;

\NI (iv) $\tilde{\clm}'=\clb(\clk)_{\tilde{\tau}}$; 

\NI (v) $\pi_{\omega}(\clb_R)'=\pi_{\omega}(\clb_L)''$ (Haag duality).

\vsp 
In such a case ( i.e. if $\omega$ is pure ) $\clm'=\tilde{\clm}$ as von Neumann 
sub-algebra of $\clb(\clk)$ and $\clm'_0= \tilde{\clm}_0$ as von-Neumann sub-algebra of $\clb(\clk_0)$. Furthermore,  $\tilde{\clm}=P\pi(\tilde{\clo}_d)''P$ and $\tilde{\clm}_0=P\pi(\tilde{\mbox{UHF}}_d)''P$.
\end{pro} 

\vsp 
\begin{proof} 
For (a)-(c) we refer to Proposition 3.1 and Proposition 3.2 in [Mo2] and for (d) we refer to Theorem 3.6 in [Mo2]. 
\end{proof}

\vsp
We consider ( [BR-II] page 110 ) the unique KMS state $\psi_{\beta}$ at inverse temperature $\beta=ln(d)$ for the automorphism $\alpha_t(s_i)=e^{it}s_i:\;t \in \IR$ on $\clo_d$. The state $\psi_{\beta}$ is $\lambda$ invariant and its restriction  $\omega_0=\psi_{|}\mbox{UHF}_d$ is the unique faithful trace on $\clb_R$. The state 
$\psi_{\beta}$ being a KMS state for an automorphism, the normal state induced by the cyclic vector on $\pi_{\psi}(\clo_d)''$ is also separating for $\pi(\clo_d)''$. Thus we check that $v^*_k=S^*_k$ and a simple computation shows that $\clj \tilde{v}^*_k \clj = {1 \over d} S_k$. Thus equality $\tilde{\clm}=\clm'$ holds. But the corner vector space $P\pi(\tilde{\clo}_d)''P$ generated by the elements $\{\tilde{v}_I\tilde{v}^*_J: |I|,|J| < \infty \}$ is equal to the linear span of $\{S_I,S^*_J: |I|,|J| < \infty\}$ which fails to be an algebra. Thus $P \neq \tilde{E}F$. However,  the unique normalize trace $\omega_0$ on $\clb$ is having symmetry given in (5) and (6) and also reflection positive. This gives 
an example indicating how pure property of $\omega$ balances the support projections of 
$\psi$ and $\tilde{\psi}$ in the amalgamated representation $\pi$ of $\clo_d$ and $\tilde{\clo}_d$ respectively.

\section{Symmetries of a translation-invariant pure state on $\clb$ } 

\vsp 
In this section we investigate translation-invariant factor states on $\clb$ with 
some discrete symmetries. Finally we specialized our analysis to pure states with additional symmetries and its two-point correlation functions.  
 
\begin{pro} 
Let $\omega$ be a lattice-symmetric translation-invariant factor state of $\clb$. Then the 
map $\psi \raro \tilde{\psi}$ is an affine bijective map on $K_{\omega}$. Furthermore,  
there exists an extremal element $\psi \in K_{\omega}$ for which 
$$\tilde{\psi}=\psi \beta_{\zeta},$$ 
where $\zeta^2 \in H$ and $\zeta=exp^{\pi i \over n}$, if $H=\{z: z^n=1 \}$ 
for some $n \ge 1$ otherwise $\zeta=1$, if $H=S^1$. 
\end{pro} 
\begin{proof} 
That the map $\psi \raro \tilde{\psi}$ is affine follows once we use implicit definition of dual states given in (28). Since $\tilde{\psi}$ restricted to $\clb_R$ is $\tilde{\omega}_R$, the map is bijective from $K_{\omega}$ onto $K_{\tilde{\omega}}$ 
as $\tilde{\tilde{\psi}}=\psi$. In particular,  the map is affine and bijective on $K_{\omega}$ once $\omega=\tilde{\omega}$. In particular,  the affine map takes any given extremal element of $K_{\omega}$ to another extremal element of $K_{\omega}$ if not itself. 

\vsp 
We fix an extremal element $\psi \in K_{\omega}$ and recall closed subgroup 
$H=\{z \in S^1: \psi = \psi \beta_z \}$ of $S^1$ is independent of choice for 
an extremal point $\psi \in K_{\omega}$ by Proposition 2.3. Thus the set of 
extremal elements in $K_{\omega}$ can be identified with $S^1 / H \equiv S^1$ 
or $\{1 \}$ by the map $z \raro \psi \beta_z$.  

\vsp 
If $\omega=\tilde{\omega}$ then, $\tilde{\psi}$ is also an extremal element 
in $K_{\omega}$ and thus by Proposition 2.3 there exists an $\zeta_0 \in S^1$ such that $\tilde{\psi}=\psi \beta_{\zeta_0}$. Since $\tilde{\tilde{\psi}}=\psi$ and $\tilde{\psi \beta_z}=\tilde{\psi} \beta_z$ for all $z \in S^1$, we compute in particular,  for 
$\omega=\tilde{\omega}$ that 
$$\psi = \tilde{\tilde{\psi}} = (\psi \beta_{\zeta_0}\tilde{)} = \tilde{\psi} \beta_{\zeta_0} =\psi \zeta_{\zeta_0^2}$$ 
Thus $\zeta_0^2 \in H$. Without loose of generality we can take 
$\zeta_0$ to be any element in the quotient group $S^1 / H$ i.e. any other 
element in $\zeta \in S^1$ for which $\zeta^{-1} \zeta_0 \in H$.   

\vsp 
If $H=\{z \in S^1: z^n=1 \}$, we can modify our choice for extremal element $\psi \in K_{\omega}$, if required, for $\tilde{\psi}=\psi \beta_{\zeta}$ with $\zeta = exp^{\pi i \over n}$ as follows: 

\vsp 
Since $(\beta \beta_z\tilde{)}=\tilde{\beta} \beta_z$ for all $z \in S^1$, the affine map $\psi \raro \tilde{\psi}$ takes $\psi \beta_z \raro \psi \beta_{\zeta_0} \beta_z $ and 
the affine map induces a continuous one to one and onto map on quotient group $S^1 / H \equiv S^1$ as $\tilde{\tilde{\psi}}=\psi$. However,  any continuous one to one and onto map on $S^1$ with its own inverse is either has a fixed point or it is a rotation map by an angle $\pi$. So either the affine map induces a fixed point in the map 
$\psi \beta_z \raro \psi \beta_{\zeta_0} \beta_z,\;z \in S^1 / H$ or 
$\zeta_0^2=1$ in $S^1 / H \equiv S^1$.   

\vsp  
Thus there exists an extremal element still denoting $\psi \in K_{\omega}$ so that $\tilde{\psi}=\psi \beta_{\zeta}$,  where $\zeta$ is either $1$ or $-1$,  where we have identified $S^1 / H =S^1$ when $H \neq S^1$. Once we remove the identification $S^1 / H \equiv S^1$ to arrive $\tilde{\psi}=\psi \beta_{\zeta}$, where $\zeta$ is either $1$ or $exp^{\pi i \over n }$.  

\vsp 
Note that in case $H=S^1$ then, $\tilde{\psi}=\psi$ for $\psi \in K_{\omega}$ as $K_{\omega}$ is a singleton set by Proposition 2.3. 
\end{proof} 

\vsp
\begin{pro} Let $\omega$ be a translation-invariant state on $\clb$ as in Proposition 2.5. 
If $\omega$ is lattice reflection symmetric then, the following holds:

\NI (a) If $H=\{z: z^n = 1 \}$ for some $n \ge 1$ then, $\tilde{\psi}=\psi \beta_{\zeta}$ for all $\psi \in K_{\omega}$,  where $\zeta$ is 
fixed either $1$ or $exp^{\pi i \over n }$ and there exists a unitary operator $U_{\zeta}$ on $\clh \otimes_{\clk} \tilde{\clh}$ so that
\be
U_{\zeta}^*=U_{\bar{\zeta}}, U_{\zeta}\Omega=\Omega,\;\;\;U_{\zeta} S_k U_{\zeta}^*= \beta_{\bar{\zeta}}(\tilde{S}_k)
\ee
for all $1 \le k \le d$. 

\vsp 
Furthermore,  if $\omega$ is also pure then, there exists a unitary operator $u_{\zeta}: \clk \raro \clk$ 
so that
\be
u_{\zeta}\Omega=\Omega, \;\;\; u_{\zeta}v_ku_{\zeta}^*= \beta_{\bar{\zeta}}(\tilde{v}_k)
\ee
for all $1 \le k \le d$ and $u_{\zeta} \clj u_{\zeta}^*=\clj,\;\;u_{\zeta} \Delta^{1 \over 2} u_{\zeta}^*=\Delta^{-{1 \over 2}}$, 
$u_{\zeta}^*=u_{\bar{\zeta}}$ and $u_{\zeta} \clm u_{\zeta}^* = \clm',\;u^*_{\zeta}\clm u_{\zeta}=\clm'$. Further if $\zeta=1$ then, $u_{\zeta}$ is self-adjoint and otherwise if $\zeta \ne 1$ then, $u_{\zeta}^{2n}$ is self adjoint.  

\NI (b) If $H=S^1$ then, $K_{\omega}$ has a unique element $\psi$ and (30) is valid with $\zeta=1$. Further if $\omega$ is also pure then, (31) is also valid with $\zeta=1$. 

\end{pro} 

\vsp
\begin{proof} 
Proposition 3.1 gives existence of an extremal element $\psi \in K_{\omega}$ so that $\tilde{\psi}=\psi \beta_{\zeta}$, where $\zeta$ is fixed either 
equal to $1$ or $exp^{\pi i \over n}$. As $(\psi \beta_z\tilde{)}=\tilde{\psi} \beta_z$ 
for all $z \in S^1$, a simple application of Proposition 2.3 says that $\tilde{\psi} = 
\psi \beta_{\zeta} $ for all extremal points in $K_{\omega}$ if it holds for one extremal element. Hence existence part in (a) is true by Krein-Millmann theorem.

\vsp
$\Omega$ is a cyclic vector for $\pi(\clo_d \otimes \tilde{\clo}_d)$ and thus we define
$U_{\zeta}:\clh \otimes_{\clk} \tilde{\clh} \raro \clh \otimes_{\clk} \tilde{\clh}$ by
$$U_{\zeta}: \tilde{S}_{I'}\tilde{S}_{J'}^*S_IS^*_J  \Omega \raro 
\beta_{\bar{\zeta}}(\tilde{S}_I\tilde{S}^*_JS_{I'}S_{J'}^*)\Omega$$
That $U_{\zeta}$ is a unitary operator follows once we verify the following steps 
by our condition that $\tilde{\psi}=\psi \beta_{\zeta}$ and Proposition 2.5 (a) as follows:
$$\langle \Omega,  \tilde{S}_{I'}\tilde{S}^*_{J'}S_{I}S^*_{J}\Omega \rangle$$
$$=\langle \Omega,  \tilde{S}_{I'}S_{I}S^*_{J}\tilde{S}^*_{J'}\Omega \rangle$$
( since $\pi(\clo_d)''$ commutes with $\pi(\tilde{\clo}_d)''$) 
$$=\langle \Omega, S_{\tilde{I'}}S_{I}S^*_{J}S^*_{\tilde{J'}}\Omega \rangle$$
(by Proposition 2.5 (a),  where we used $v_I^*\Omega=\tilde{v}_{\tilde{I}}\Omega$ ) 
$$=\langle \Omega,\beta_{\bar{\zeta}}( \tilde{S}_{\tilde{I'}}\tilde{S}_{I}\tilde{S}^*_{J}\tilde{S}^*_{\tilde{J'}})\Omega \rangle$$
( since $\tilde{\psi} = \psi \beta_{\zeta}$ )  
$$=\langle \Omega, \beta_{\bar{\zeta}}(S_{I'}\tilde{S}_I\tilde{S}^*_JS_{J'}^*)\Omega \rangle $$
(by Proposition 2.5 (a) again) 
$$=\langle \Omega, \beta_{\bar{\zeta}}(\tilde{S}_I\tilde{S}^*_JS_{I'}S_{J'}^*)\Omega \rangle$$
(since $\pi(\clo_d)''$ commutes with $\pi(\tilde{\clo}_d)''$)

\vsp 
By Cuntz relation (9),  we have 
$$(S_{I_1}S^*_{J_1})^*S_{I_2}S^*_{J_2}$$
$$=S_{J_1}S_{I_1}^*S_{I_2}S^*_{J_2}$$
$$=0 \;\mbox{or}\;S_{I_3}S^*_{J_3} \neq 0$$ 
, where $I_3$ and $J_3$ in $\cli$ are uniquely determined by 
sets $I_1,J_1$ and $I_2,J_2$. Note also that 
$\beta_{g_0}( (S_{I_1}S^*_{J_1})^*)\beta_{g_0}(S_{I_2}S^*_{J_2})=\beta_{g_0}(S_{I_3}S^*_{J_3})$. Note that the same set $I_3,J_3$ is determined if we use Cuntz elements $\pi(\tilde{\clo}_d)$ instead of $\pi(\clo_d)$. Now we verify the inter-product 
preserving property between two elements of the form $f_{I',J',I,J}=\tilde{S}_{I'}\tilde{S}^*_{J'}S_{I}S^*_{J}\Omega$ as follows:
$$\langle f_{I_1',J_1',I_1.J_1}, f_{I'_2,J'_2,I_2,J_2} \rangle$$
$$=\langle \Omega, f_{I'_3,J'_3,I_3,J_3} \rangle$$
$$=\langle \Omega, \beta_{g_0}(f_{I_3,J_3,I'_3,J'_3}) \rangle$$
$$=\langle \beta_{g_0}(f_{I_1,J_1,I'_1,J'_1}), \beta_{g_0}(f_{I_2,J_2,I'_2,J'_2}) \rangle$$ 

\vsp 
By our construction we also have 
$U_{\zeta}S_k = \beta_{\bar{\zeta}}(\tilde{S}_k)U_{\zeta}$ for all $1 \le k \le d$. In particular,  $U_{\zeta}\pi(\clo_d)''U_{\zeta}^*=\pi(\tilde{\clo}_d)''$. 

\vsp 
We recall projections $E=[\pi(\clo_d)'\Omega]$ and $\tilde{E}=[\pi(\tilde{\clo}_d)'\Omega]$  
defined in Proposition 2.5 (d). We also recall projections $F=[\pi(\clo_d)'\Omega]$ and $\tilde{F}=[\pi(\tilde{\clo}_d\Omega]$. We have $P=EF$ in general by our construction 
of amalgamated representation. If the state $\omega$ is pure, then by Proposition 2.5 (d) we have $F=\tilde{E}$ and $\tilde{F}=E$. Thus  
$$EF = E \tilde{E}=\tilde{E}\tilde{F}$$ 
i.e. $P=E\tilde{E}$. 

\vsp 
Once the pure state $\omega$ is also lattice reflection-symmetric, by our construction of 
$U_{\zeta}$, we have $U_{\zeta}\pi(\clo_d)'U^*_{\zeta}=\pi(\tilde{\clo}_d)'$ and $U_{\zeta}\pi(\tilde{\clo}_d)'U^*_{\zeta} = \pi(\clo_d)'$. Furthermore,  $U_{\zeta}\Omega=\Omega$ says that $U_{\zeta} E U_{\zeta}^*=\tilde{E}$ and 
$U_{\zeta} \tilde{E} U_{\zeta}^*=E$. Thus we have  
$U_{\zeta}PU_{\zeta}^*=U_{\zeta}E \tilde{E}U_{\zeta}^*=\tilde{E}E=P$, which ensures a unitary operator $u_{\zeta}=PU_{\zeta}P$ on $\clk$. Furthermore,  a routine calculation 
shows using (26) that
\be
u_{\zeta}\Omega=\Omega,\; u_{\zeta}v^*_k u_{\zeta}^*= \beta_{\bar{\zeta}}(\tilde{v}^*_k)
\ee
for all $1 \le k \le d$. As $U_{\zeta}^*=U_{\bar{\zeta}}$,  we have $u_{\zeta}^*=u_{\bar{\zeta}}$. If $\zeta \ne 1$, then
$\zeta^{2n}=1$ and thus $U^{2n}_{\zeta}$ is inverse of its own. Thus $u^{2n}_{\zeta}$ is self-adjoint. By Proposition 2.5 (d),  we have in general $\clm'=\tilde{\clm}$ if 
$\omega$ is pure. 

\vsp
We recall Tomita's conjugation operators [BR-I] 
$Sx\Omega=x^*\Omega,\;x \in \clm$ and $Fx'\Omega=x'^*\Omega,\;x' \in \clm'$ and verify 
the following equalities using (28) for $|I|,|J| < \infty$ 
$$u_{\zeta}Sv_Iv_J^*\Omega$$
$$=u_{\zeta}v_Jv_I^*\Omega$$
$$=u_{\zeta}v_Jv_I^*u^*_{\zeta}\Omega$$
$$=\beta_{\bar{\zeta}}(\tilde{v}_J\tilde{v}^*_I)\Omega$$
$$=F\beta_{\bar{\zeta}}(\tilde{v}_I\tilde{v}^*_J)\Omega$$
$$=Fu_{\zeta}v_Iv_J^*u_{\zeta}^*\Omega$$
$$=Fu_{\zeta}v_Iv_J^*\Omega$$ 
Since $F=\clj \Delta^{1 \over 2}$ on $\clm\Omega$ and $F=\clj \Delta^{-{1 \over 2}}$ on $\clm'\Omega$,  we have 
$u_{\zeta} \clj \Delta^{1 \over 2} = \clj \Delta^{-{1 \over 2}}u_{\zeta}$ on $\clm \Omega$, i.e $u_{\zeta} \clj u^*_{\zeta} 
u_{\zeta}\Delta^{1 \over 2}u^*_{\zeta} = \clj \Delta^{-{1 \over 2}}$ on $\clm \Omega$. 
Thus by the uniqueness of polar decomposition of $S$ we 
conclude that $u_{\zeta} \clj u^*_{\zeta} = \clj$ and $u_{\zeta}\Delta^{1 \over 2}u^*_{\zeta} = \Delta^{-{1 \over 2}}$. 

\vsp   
By Proposition 2.5 (d) we get 
$$u_{\zeta}\clm u^*_{\zeta}$$
$$=PU_{\zeta}P\pi(\clo_d)''PU_{\zeta}^*P$$
$$=PU_{\zeta}\pi)\clo_d)''U_{\zeta}^*P$$
$$=P\pi(\tilde{\clo}_d)''P$$
$$=\tilde{\clm}$$ 
Along the same line we also have $u\tilde{\clm}u^*=\clm$ i.e. 
$u^*\clm u=\tilde{\clm}$.

\vsp 
Since $H=S^1$, non empty set $K_{\omega}$ contains only one element and so $\tilde{\psi}=\psi$. Rest of the proof goes along the same line as of (a) without the phase factor involved in (26) and (27).      
\end{proof}

\vsp 
We make a simple observation here that $H=S^1$ for the unique KMS-state $\psi$ for the automorphisms $\alpha_t(s_k)=e^{it}s_k$ on $\clo_d$ Proposition 3.2 says that 
(26) is valid with $\zeta=1$. It is easy to check that (27) is not valid as 
$v^*_k=S^*_k$ and $\clj \tilde{v}^*_k \clj = {1 \over d}S_k$ for all $1 \le k \le d$. 
For details we refer to [BR-II] (page-110).

\vsp
Now we introduce another useful symmetry on $\omega$. If $Q= Q^{(l)}_0 \otimes Q^{(l+1)}_1 \otimes ....\otimes Q^{(l+m)}_m$ 
we set $Q^t={Q^t}^{(l)}_0 \otimes {Q^t}^{(l+1)}_1 \otimes ..\otimes {Q^t}^{(l+m)}_m$
, where $Q_0,Q_1,...,Q_m$ are arbitrary elements in $\!M_d$ and $Q_0^t,Q^t_1,..$ stands for transpose
with respect to an orthonormal basis $(e_i)$ for $\IC^d$ (not complex conjugate) of $Q_0,Q_1,..$ 
respectively. We define $Q^t$ by extending linearly for any
$Q \in \clb_{loc}$. For a state $\omega$ on $\clb$ we define
a state $\bar{\omega}$ on $\clb$ by the following prescription
\be
\bar{\omega}(Q) =
\omega(Q^t)
\ee
Thus the state $\bar{\omega}$ is a translation-invariant, ergodic, factor state,  if and only if $\omega$ is a translation-invariant, ergodic, factor state respectively. We say 
$\omega$ is {\it real } if $\bar{\omega}=\omega$. In this section we study a translation-invariant real state.

\vsp
For a $\lambda$ invariant state $\psi$ on $\clo_d$ we define a $\lambda$ invariant state $\bar{\psi}$ by
\be
\bar{\psi}(s_Is^*_J)=\psi(s_Js^*_I)
\ee
for all $|I|,|J| < \infty$ and extend linearly. That it defines a state follows as 
for an element $x= \sum c(I,J)s_Is^*_J$,  we have $\bar{\psi}(x^*x)=\psi(y^*y) \ge 0$ 
, where $y= \sum \overline{c(I,J)}s_Js_I^*$. It is obvious that $\psi \in K_{\omega}$,  if and only if $\bar{\psi}
\in K_{\bar{\omega}}$ and the map $\psi \raro \bar{\psi}$ is an affine map. In particular,  an
extremal point in $K_{\omega}$ is also mapped to an extremal point in $K_{\bar{\omega}}$. It is also
clear that $\bar{\psi} \in K_{\omega}$,  if and only if $\omega$ is real. Hence a real
state $\omega$ determines an affine bijective map $\psi \raro \bar{\psi}$ on the compact convex set $K_{\omega}$. Furthermore, if $\omega$ is also extremal on $\clb$, then the affine map takes an extremal element $\psi_0$ to another extremal element $\bar{\psi_0}$ 
of $K_{\omega}$ and there exists a $\zeta_0 \in S^1$ so that $\bar{\psi}_0 = \psi_0 \beta_{\zeta_0}$. However,  $\zeta_0$ is not uniquely determined and for any other 
$\zeta \in S^1$ for which $\zeta \zeta_0^{-1} \in H$, we also have $\tilde{\psi_0}=\psi \beta_{\zeta}$. The set of extremal elements in $K_{\omega}$ can be identified with $S^1 / H \equiv S^1$ or $\{ 1 \}$ via the map $z \raro \psi_0 \beta_z$.  

\vsp 
Since $\overline{\psi_0 \beta_z}= \bar{\psi}_0 \beta_{\bar{z}}$ for all $z \in S^1$ and the affine map $\psi \raro \bar{\psi}$ takes $\psi_0 \beta_z$ to $\psi_0 \beta_{z_0 \bar{z}}$. If $z_0=1$ the map fixes two-point namely $\psi_0$ and $\psi_0 \beta_{-1}$. Even otherwise i.e. $z_0 \neq 1$ we can choose $z \in S^1$ so that $z^2=z_0$ and verify easily that the affine map have a fixed point $\psi_0 \beta_z$. What is also crucial here that we can 
as well choose $z \in S^1$ so that $z^2 =-z_0$, if so, then 
$\psi_0 \beta_z $ gets mapped into $\psi_0 \beta_{z_0} \beta_{\bar{z}} = \psi_0 \beta_{-z}=\psi_0 \beta_z \beta_{-1}$. Thus in any case for $\zeta \in \{1,-1\}$, we also have an extremal element $\psi \in K_{\omega}$ so that $\bar{\psi} = \psi \beta_{\zeta}$.    

\vsp
Thus going back to the original set up, we sum up the above by saying that if  $H=\{z: z^n= 1 \} \subset S^1$ and 
$\zeta \in \{1,exp^{i\pi \over n } \}$ then there exists an extremal element $\psi \in K_{\omega}$ so that $\bar{\psi}
=\psi \beta_{\zeta}$.   

\vsp
\begin{pro} 
Let $\omega$ be a translation-invariant real factor state on $\otimes_{\IZ}M_d$. Then the 
following holds: 

\NI (a) if $H=\{z: z^n= 1 \} \subseteq S^1$ and $\zeta \in \{1,exp^{i\pi \over n } \}$ then there exists an extremal 
element $\psi \in K_{\omega}$ so that $\bar{\psi}=\psi \beta_{\zeta}$. Let $(\clh,\pi_{\psi}(s_k)=S_k,\;1 \le k \le d, 
\Omega)$ be the GNS representation of $(\clo_d,\psi)$, $P$ be the support projection of the state $\psi$ in $\pi(\clo_d)''$ 
and $(\clk,\clm,v_k,\;1 \le k \le d, \Omega)$ be the associated Popescu's systems as in Proposition 2.1. Let $\bar{v}_k= \clj v_k \clj$ 
for all $1 \le k \le d$ and $(\bar{\clh},\bar{S}_k,P,\Omega)$ be the Popescu's minimal dilation as described in the converse part of Proposition 2.1 (Theorem 2.1 in [Mo2]) associated with the systems $(\clk,\clm',\bar{v}_k,\;1 \le k \le d,\Omega)$. Then 
there exists a unitary operator
$W_{\zeta}: \clh \raro \bar{\clh}$ so that
\be
W_{\zeta}\Omega=\Omega,\;\;\;W_{\zeta}S_kW_{\zeta}^* = \beta_{\bar{\zeta}}(\bar{S}_k) 
\ee
for all $1 \le k \le d$. Furthermore,  $P$ is the support projection of the state $\bar{\psi}$ in $\bar{\pi}(\clo_d)''$ 
and there exists a unitary operator $w_{\zeta}$ on $\clk$ so that
\be
w_{\zeta}\Omega=\Omega, \;\;\; w_{\zeta}v_kw_{\zeta}^* = \beta_{\bar{\zeta}}(\bar{v}_k)= \clj \beta_{\zeta}(v_k) \clj
\ee
for all $1 \le k \le d$ and  $w_{\zeta} \clj w_{\zeta}^*=\clj$ and $w_{\zeta} \Delta^{1 \over 2} w_{\zeta}^* = 
\Delta^{-{1 \over 2}}$. $w_{\zeta}$ is self adjoint,  if and only if $\zeta=1$; 

\NI (b) If $H=S^1$, $K_{\omega}$ is a set with unique element $\psi$ so that $\bar{\psi}=\psi$ and relations 
(35) and (36) are valid with $\zeta=1$.   
\end{pro}

\vsp
\begin{proof} 
For existence part in (a) we refer to the paragraph above preceded the statement of the proposition. 
We fix a state $\psi \in K_{\omega}$ so that $\bar{\psi}=\psi  \beta_{\zeta}$ and define $W:\clh \raro \bar{\clh}$ by
$$W_{\zeta}: S_IS^*_J\Omega \raro \beta_{\bar{\zeta}}(\bar{S}^*_{I}\bar{S}^*_{J})\Omega$$
That $W_{\zeta}$ is a unitary operator follows from (30) and thus
$W_{\zeta}S_k= \beta_{\bar{\zeta}}(\bar{S}_k)W_{\zeta}$ for all $1 \le k \le d$. 

\vsp 
By Proposition 2.1, $P$ being the support projection of $\psi$ in $\pi(\clo_d)''$, we have $\clm'= \{ x \in \clb(\clh): \sum_k v_k x v_k^*= x \}$. 
Since $\clj \clm' \clm =\clm$, we also have $\clm = \{ x \in \clb(\clk): \sum_k \clj v_k \clj x \clj v_k^* \clj = x \}$. Hence by the converse part of Proposition 2.1, we conclude that $P$ is also the support projection of the state $\bar{\psi}$ in $\bar{\pi}(\clo_d)''$. 
Hence $W_{\zeta}PW_{\zeta}^*=P$. Thus we define a unitary operator $w_{\zeta}: \clk \raro \clk$ by
$w_{\zeta}= PW_{\zeta}P $ and verify that
$$\bar{v}^*_k = P\bar{S}^*_kP$$
$$ = PW_{\zeta} \beta_{\zeta}(S^*_k)W_{\zeta}^*P = PW_{\zeta}P \beta_{\zeta}(S^*_k)PW_{\zeta}^*P$$
$$= PW_{\zeta}P\beta_{\zeta}(v^*_k)PW_{\zeta}^*P=w_{\zeta}\beta_{\zeta}(v^*_k)w_{\zeta}^*.$$

\vsp
We recall that Tomita's conjugate linear operators $S,F$ [BR-I] are
the closure of the linear operators defined by $S_0:a\Omega \raro a^*\Omega$ for $a \in \clm$ and $F_0:b\Omega \raro b^*\Omega$ for $y \in \clm'$. We check the following equalities
$$w_{\zeta}Sv_Iv^*_J\Omega= w_{\zeta} v_Jv^*_I\Omega= w_{\zeta}v_Jv^*_Iw^*_{\zeta}\Omega$$
$$=\zeta^{|I|-|J|} \bar{v}_J \bar{v}^*_I\Omega= \zeta^{|I|-|J|} F \bar{v}_I\bar{v}^*_J\Omega$$
$$=F\zeta^{-|I|+|J|} \bar{v}_I\bar{v}^*_J\Omega$$
$$=F w_{\zeta} v_Iv^*_J\Omega$$
for all $|I|,|J| < \infty$. Thus $w_{\zeta}Sw^*_{\zeta}=F$ on the domain of $F$. 
We recall the unique polar decomposition of $S$ given by $S=\clj \Delta^{1 \over 2}$ and 
$F=S^*=\Delta^{1 \over 2} \clj= \clj \Delta^{-{1 \over 2}}$. Hence $w_{\zeta} \clj w_{\zeta}^*w_{\zeta}\Delta^{1 \over 2}w^*_{\zeta}=\clj \Delta^{-{1 \over 2}}$.
By the uniqueness of polar decomposition of $S$ we get $w_{\zeta} \clj w^*_{\zeta}=\clj$ and $w_{\zeta} \Delta^{1 \over 2} w^*_{\zeta}= \Delta^{-{1 \over 2}}$.

\vsp  
Since $\clj$ commutes with $w_{\zeta}$ we also have  
$$w_{\zeta}^2 v_k (w_{\zeta}^*)^2$$
$$=w_{\zeta} \clj \beta_{\zeta}(v_k) \clj w_{\zeta}^*$$
$$=\clj w_{\zeta} \beta_{\zeta}(v_k) w_{\zeta}^* \clj$$
$$=\clj \clj \beta_{\zeta^2}(v_k) \clj \clj$$
$$=\beta_{\bar{\zeta}^2}(v_k)$$ 

\vsp 
Further $\zeta^2=1$,  if and only if $\zeta=1$ ( as $\zeta=1$ or $exp^{i \pi \over n}$,  where $n \ge 2$ ). In such a case we get $w_{\zeta}^2 \in \clm'$ and further 
as $w_{\zeta}$ commutes with $\clj$, $w_{\zeta}^2 =\clj w_{\zeta}^2 \clj \in \clm$. $\omega$ being an extremal element in $K_{\omega}$,  we have $\clm \vee \tilde{\clm}=\clb(\clk)$ by Proposition 2.5 (c) and as $\tilde{\clm} \subseteq \clm'$, we get that $\clm$ is a factor. Thus for a factor $\clm$, $w_{\zeta}^2$ is a scalar. Since 
$w_{\zeta}\Omega=\Omega$ we get $w_{\zeta}^2=1$ i.e. $w^*_{\zeta}=w_{\zeta}$. This completes the proof. 
\end{proof} 

\vsp
\begin{thm} 
Let $\omega$ be a translation-invariant factor state on $\clb$. Then the following 
are equivalent:

\NI (a) $\omega$ is real and lattice reflection-symmetric;

\NI (b) There exists an extremal element $\psi \in K_{\omega}$ so that $\tilde{\psi}=\psi \beta_{\zeta}$ and 
$\bar{\psi}=\psi \beta_{\zeta}$, where $\zeta$ is either $1$ or $exp^{i\pi \over n }$.    

\vsp
Furthermore,  if $\omega$ is a pure state then the following holds:

\NI (c) There exists a family of a Popescu's element $(\clk,v_k,\;1 \le k \le d,\Omega)$ for an extremal element $\psi \in K_{\omega}$ with relation  
\be 
v_k = \clj_v \tilde{v}_k \clj_v 
\ee 
for all $1 \le k \le d$, where $\clj_v=v\clj$ and $v$ is a self-adjoint unitary 
operator on $\clk$ commuting with modular operators $\Delta^{1 \over 2}$ and conjugate operator $\clj$ associated 
with cyclic and separating vector $\Omega$ for $\clm$. Furthermore,  $\beta_z(v)=v$ for all $z \in H$ and  
$H \subseteq \{1, -1 \}$;

\NI (d) The map $\clj_v: \clh \otimes_{\clk} \tilde{\clh} \raro \clh \otimes_{\clk} \tilde{\clh}$ defined by 
$$\pi(s_Is^*_J\tilde{s}_{I'}\tilde{s}^*_{J'})\Omega 
\raro \pi(s_{I'}s^*_{J'}\tilde{s}_{I}\tilde{s}^*_{J})\Omega,$$
$\;|I|,|J|,|I'|,|J'| < \infty $ extends the map $\clj_v: \clk  \raro  \clk$ to an anti-unitary map so that $\clj_v 
\pi(s_i) \clj_v = \bar{\pi}(\tilde{s}_i)$ for all $ 1 \le i \le d$,  where
$\bar{\pi}$ is the conjugate linear extension of $\pi$ from the generating 
set $(\tilde{s}_i)$, i.e. $\bar{\pi}(\tilde{s}_I\tilde{s}^*_J) = \pi(\tilde{s}_I\tilde{s}^*_J)$ for $|I|,|J| < \infty$ 
and then extend it anti-linearly for its linear combinations. 
\end{thm} 

\vsp
\begin{proof} 
Since $\omega$ is lattice reflection-symmetric, by Proposition 3.2 $\tilde{\psi}= \psi \beta_{\zeta}$ 
for all $\psi \in K_{\omega}$,  where $\zeta$ is a fixed number taking values in 
$\{1,exp^{i \pi \over n}\}$ if $H=\{z:z^n=1 \}$ for some $n \ge 1$. Now we use real property of $\omega$ and choose by Proposition 3.3 an extremal element $\psi \in K_{\omega}$ so that $\bar{\psi}=\psi \beta_{\zeta}$. This proves that (a) implies (b). 
That (b) implies (a) is obvious.

\vsp
Now we aim to prove the last statements which is the main point of the proposition. We 
consider the Popescu's element $(\clk,\clm,v_k,\;1 \le k \le d,\Omega)$ as in Proposition 3.2. Thus by Proposition 3.2 and Proposition 3.3 there exists unitary operators $u_{\zeta},w_{\zeta}$ on $\clk$ so that
$$u_{\zeta}v_ku_{\zeta}^*=\beta_{\bar{\zeta}}(\tilde{v}_k)$$
$$w_{\zeta}v_kw_{\zeta}^* = \beta_{\bar{\zeta}}(\bar{v}_k) = \clj \beta_{\zeta}(v_k) \clj $$
, where $\;\;u_{\zeta} \clj u_{\zeta}^*=\clj, \;\; 
\;w_{\zeta} \clj w_{\zeta}^*=\clj$ and
$u_{\zeta} \Delta^{1 \over 2}u_{\zeta}^*=w_{\zeta} \Delta^{1 \over 2}w^*_{\zeta}=\Delta^{-{1 \over 2}}$.
Thus
\be
u_{\zeta} w_{\zeta} v_k w_{\zeta}^* u_{\zeta}^* = u_{\zeta} \clj \beta_{\zeta}(v_k) \clj u_{\zeta}^* 
= \clj \beta_{\zeta}(u_{\zeta}v_ku_{\zeta}^*) \clj = \clj \beta_{\zeta}(\beta_{\bar{\zeta}}(\tilde{v}_k))\clj=
\clj \tilde{v}_k \clj
\ee

We also compute that
\be
w_{\zeta}u_{\zeta}v_ku_{\zeta}^*w_{\zeta}^* = w_{\zeta} \beta_{\bar{\zeta}}(\tilde{v}_k) w_{\zeta}^* = 
\clj \beta_{\zeta} \beta_{\bar{\zeta}}(\tilde{v}_k) \clj = \clj \tilde{v}_k \clj 
\ee

\vsp
By Proposition 2.5 (c), for a factor state $\omega$ we also have $\clm \vee \tilde{\clm} = \clb(\clk)$. As $\tilde{\clm} \subseteq \clm'$, in particular,  
we note that $\clm$ is a factor. So $u_{\zeta}^*w_{\zeta}^* u_{\zeta}w_{\zeta} \in \clm'$ commuting also with $\clj$ and thus a scalar as $\clm$ is a factor. 
As $u_{\zeta}\Omega=w_{\zeta}\Omega=\Omega$, we conclude that $u_{\zeta}$ commutes with $w_{\zeta}$.  

\vsp
Now we set $v_{\zeta}=u_{\zeta}w_{\zeta}$ which is a unitary operator commuting with both $\clj$ and
$\Delta^{1 \over 2}$. That $v_{\zeta}$ is commuting with $\Delta^{1 \over 2}$ follows as $u_{\zeta}w_{\zeta}\Delta^{1 \over 2}=
u_{\zeta}\Delta^{- {1 \over 2}}w_{\zeta}= \Delta^{1 \over 2} u_{\zeta}w_{\zeta}$. 

\vsp
We claim now that $v_{\zeta}$ is a self-adjoint element. To that end we note the relation
(38) says that $v_{\zeta}\clm v_{\zeta}^* \subseteq \clm$ and so 
$v_{\zeta}\clm'v_{\zeta}^* \subseteq \clm'$. Now we check the following identity:
$v_{\zeta}\tilde{v}^*_kv_{\zeta}^*\Omega
=v_{\zeta}v_k^*\Omega=v_{\zeta}v_k^*v_{\zeta}^*\Omega
=\clj \tilde{v}_k^* \clj \Omega
=\clj v_k \clj \Omega$
for all $1 \le k \le d$. Thus by separating property of $\Omega$ for $\clm'$ we deduce that 
$$v_{\zeta}\tilde{v}_k^*v_{\zeta}^*=\clj v_k \clj$$
for all $1 \le k \le d$. So we conclude that $v^2_{\zeta} \in \clm'$ and as $v_{\zeta}$ commutes with $\clj$, $v_{\zeta}^2$ is an element in the 
centre of $\clm$. The centre of $\clm$ being trivial as $\omega$ is a factor state ( here we have more namely pure )
and $v_{\zeta}\Omega=\Omega$, we conclude that $v^2_{\zeta}$ is the unit operator. Hence $v_{\zeta}$ is a self-adjoint element. 

\vsp 
For simplicity of notation we set $v$ for $v_{\zeta}$ for the rest of the proof and 
verify now that $\beta_z(v)=v$ for all $z \in H$ by proving that $v$ commutes 
$P_k$ for each $k \in \hat{H}$. Since $vv_Iv_J^*v^*=\clj \tilde{v}_I\tilde{v}^*_J\clj$ 
with $v\Omega=\Omega$, we get 
$$vv_Iv_J^*\Omega=\clj \tilde{v}_I\tilde{v}^*_J\Omega$$
for all $I|,|J| < \infty$. However,  we recall that $P_k=\{f \in \clk: u_zf=z^kf \}$ 
which is equal to $[v_Iv^*_J\Omega: |I|-|J|=k]$ since $[\clm\Omega]=\clk$ by our construction in Proposition 2.1. It is clear also that the projection  $\tilde{P}_k=[\tilde{v}_I\tilde{v}^*_J\Omega: |I|-|J|=k ] \subseteq P_k$ for all $k \in \hat{H}$ since $(u_z)$ commutes with $\clj$ and $\Delta^{1 \over 2}$,  we have
$\beta_z(\tilde{v}_i)=z\tilde{v}_i$ for all $1 \le i \le d$. However,  $\omega$ being 
pure,  we have $\tilde{\clm}=P\pi(\tilde{\clo}_d)P$ and $\tilde{\clm}=\clm'$ by Proposition 2.5 (d). So $\sum_k \tilde{P}_k =[\tilde{\clm}\Omega]=I_{\clk}$. Thus $P_k=\tilde{P}_k$ 
for all $k \in \hat{H}$.  

\vsp 
So we get $(I-P_k)vP_k=0$ as $\clj$ commutes with $P_k$. We can interchange the role of 
$(v_i)$ and $(\tilde{v}_i)$ in order to conclude 
that $(I-P_k)v^*P_k=0$ for all $k \in \hat{H}$. Otherwise we can as well use 
the fact that $v$ is self-adjoint to conclude that $v$ commutes with each $P_k$. 
This shows that $v$ commutes with $u_z$ for all $z \in H$ i.e. $\beta_z(v)=v$. 

\vsp 
Fix any $z \in H$. By taking action of $\beta_z$ on both sides of the relation $vv_kv^*=\clj \tilde{v}_k \clj$, we have 
$vv_kv^* = \bar{z}^2  \clj \tilde{v}_k \clj  = \bar{z}^2 vv_kv^*$. Thus $z^2v_k=v_k$ for all 
$1 \le k \le d$. Since $\sum_k v_kv_k^*=1$, we have $z^2=1$. Thus $H \subseteq \{-1,1 \}$.   

\vsp 
The last statement (d) follows by a routine calculation as shown below for a special vectors.
$$ \langle \Omega,\pi( \tilde{s}_{I'}\tilde{s}^*_{J'}s_Is^*_J )\Omega \rangle $$
$$= \langle \Omega,  \tilde{v}_{I'}\tilde{v}^*_{J'} v_Iv_J^* \Omega \rangle $$
$$= \langle \Omega,\clj_v v_{I'}v^*_{J'} \tilde{v}_I\tilde{v}_J^* \clj _v\Omega \rangle  $$ 
( as $\clj_v v_i \clj_v = \tilde{v}_i$)
$$= \langle v_{I'}v^*_{J'} \tilde{v}_I\tilde{v}_J^*\Omega,\Omega \rangle $$ ($\clj_v$ being anti-unitary )
$$= \langle \tilde{v}_I\tilde{v}_J^* v_{I'}v^*_{J'} \Omega,\Omega \rangle$$
($\tilde{\clm} \subseteq \clm'$) 
$$= \langle \pi(\tilde{s}_I\tilde{s}^*_Js_{I'}s^*_{J'})\Omega,\Omega \rangle $$ 

\vsp 
For anti-unitary relation involving more general vectors, we use Cuntz relations (10) as in Proposition 3.2 to reduce the inner product between two elements to the above special 
case. The last statement is obvious as $\clj_v$ is anti-linear. This completes the proof.
\end{proof} 

\vsp 
Now we aim to deal with states $\omega$ having reflection symmetry with a twist $g_0 \in U_d(\IC)$ introduced in [FILS]. To that end we fix any $g_0 \in U_d(\IC)$ so that $g_0^2=1$ and $\beta_{g_0}$ is the natural action on $\clo_d$ and $\tilde{\clo}_d$. We say $\omega$ is {\it lattice reflection symmetric with twist } $g_0$ 
if $\omega(\beta_{g_0}(r(x))=\omega(x)$ for all $x \in \clb$,  where $r$ is the reflection automorphism around ${-{1 \over 2}}$. So when 
$g_0=1$ we get back to our notion of lattice reflection symmetric. We fix now such a lattice reflection 
$g_0$-twisted factor state $\omega$. Since $\beta_{g_0}\beta_z=\beta_z \beta_{g_0}$ for all $z \in S^1$, 
by going along the same line as in Proposition 3.2, any extremal element in $\psi$ in $K_{\omega}$ will 
admit $\tilde{\psi}^{g_0}=\psi \circ \zeta$,  where $\zeta=1$ or $\zeta = exp^{\pi i \over n }$ 
, where $H=\{z \in S^1: z^n=1 \}$ and $\tilde{\psi}^{g_0} = \tilde{\psi} \beta_{g_0}$. Thus we can follow the same
steps as in the proof of Proposition 3.4 to have modified statements of Proposition 3.4 with $v_k$ replaced by $\beta_{g_0}(v_k)$ for such a pure real state i.e. there exists unitary operators $u_{\zeta},w_{\zeta}$ on $\clk$ 
so that
$$u_{\zeta}\beta_{g_0}(v_k)u_{\zeta}^*=\beta_{\bar{\zeta}}(\tilde{v}_k)$$
$$w_{\zeta}v_kw_{\zeta}^* = \beta_{\bar{\zeta}}(\bar{v}_k) = \clj \beta_{\zeta}(v_k) \clj $$
, where $\;\;u_{\zeta} \clj u_{\zeta}^*=\clj, \;\; 
\;w_{\zeta} \clj w_{\zeta}^*=\clj$ and
$u_{\zeta} \Delta^{1 \over 2}u_{\zeta}^*=w_{\zeta} \Delta^{1 \over 2}w^*_{\zeta}=\Delta^{-{1 \over 2}}$.

Thus
\be
u_{\zeta} w_{\zeta} v_k w_{\zeta}^* u_{\zeta}^* = u_{\zeta} \clj \beta_{\zeta}(v_k) \clj u_{\zeta}^* 
= \clj \beta_{\zeta}(u_{\zeta}v_ku_{\zeta}^*) \clj = \clj \beta_{\zeta}(\beta_{g_0}(\beta_{\bar{\zeta}}(\tilde{v}_k)))\clj= \clj \beta_{g_0}(\tilde{v}_k) \clj
\ee

We also compute that
\be
w_{\zeta}u_{\zeta}v_ku_{\zeta}^*w_{\zeta}^* = w_{\zeta} \beta_{\bar{\zeta}}(\beta_{g_0}(\tilde{v}_k)) w_{\zeta}^* 
=\beta_{\bar{\zeta}}(\beta_{g_0}(w_{\zeta}\tilde{v}_k w_{\zeta}^*)) 
=\clj \beta_{\zeta} \beta_{\bar{g_0}}(\beta_{\bar{\zeta}}(\tilde{v}_k)) \clj 
= \clj \beta_{\bar{g_0}}(\tilde{v}_k) \clj 
\ee

Thus taking $v_{g_0}=w_{\zeta}u_{\zeta}$, as $g_0^2=I$ we also have 
\be 
v_{g_0}\beta_{g_0}(v_k)v^*_{g_0} = \clj \tilde{v}_k \clj
\ee 
for all $1 \le k \le d$,  where $v_{g_0}$ is a unitary operator 
commuting with $\Delta^{1 \over 2}$ and $\clj$. Now we check the following identities:
$$v_{g_0}\beta_{g_0}(\tilde{v}_k^*)v^*_{g_0}\Omega$$
$$=v_{g_0}\beta_{g_0}(v_k^*)v^*_{g_0}\Omega\;\; (\mbox{since}\;\; \tilde{v}_k^*\Omega=v_k^*\Omega)$$
$$=\clj \tilde{v}_k^*\clj\Omega\;\;(\mbox{by the relation}\; (42) )$$
$$=\clj v_k^* \clj\Omega\;\;(\mbox{by the relation}\; (25) )$$ 
Since $v_{g_0} \clm v_{g_0}^* = \clm$, we get $v_{g_0}\clm'v_{g_0}^*=\clm'$ and 
so the separating property of $\Omega$ for $\clm'$ and the above identities say that 
$$v_{g_0}\beta_{g_0}(\tilde{v}_k)v_{g_0}^*=\clj v_k \clj$$ 
for all $1 \le k \le d$. 

\vsp 
Unlike the twist free case, $v_{g_0}$ need not be self-adjoint in general. In fact 
we get the following identities:  
$$v_{g_0}^2v_k(v^*_{g_0})^2=v_{g_0} \clj \beta_{\bar{g_0}}(\tilde{v}_k) \clj v^*_{g_0}$$
$$=\clj \beta_{\bar{g_0}} (\clj \beta_{\bar{g_0}}(v_k) \clj)\clj=\beta_{\bar{g_0}g_0}(v_k)$$   
Thus $v_{g_0}$ is self adjoint,  if and only if $g_0=\bar{g_0}$ as $g_0^2=1$.  

\vsp 
Nevertheless,  we have $\beta_z(v_{g_0})=v_{g_0}$ for all $z \in H$ even if $\omega$ is  reflection symmetric with a twist $g_0$. For a proof,  we can follow the same steps that we did for the twist free case. For the sake of completeness in the following we give details. 

\vsp 
Since $v_{g_0}\beta_{g_0}(v_Iv_J^*)v^*_{g_0}=\clj \tilde{v}_I\tilde{v}^*_J\clj$ 
and $v^*_{g_0}\Omega=\Omega$ we get 
$$v_{g_0}\beta_{g_0}(v_Iv_J^*)\Omega=\clj \tilde{v}_I\tilde{v}^*_J\Omega$$
Furthermore,  since $[v_Iv^*_J\Omega: |I|-|J|=k]=P_k$ and also $\omega$ being pure  $[\tilde{v}_I\tilde{v}^*_J\Omega: |I|-|J|=k ]=P_k$ (see details in Proposition 3.4)  
we get 
$$(I-P_k)v_{g_0}P_k \beta_{g_0}(v_Iv_J^*)\Omega$$
$$(I-P_k)v_{g_0}P_k \beta_{g_0}(v_Iv_J^*)v_{g_0}^*\Omega$$
$$=(1-P_k)\clj \tilde{v}_I\tilde{v}^*_J \clj\Omega$$
$$=(I-P_k)\clj P_k \tilde{v}_I\tilde{v}_J^*\Omega$$
$$=0$$ 
for $|I|-|J|=k$ since $\clj$ commutes with $P_k$. We can interchange the 
role of $(v_i)$ and $(\tilde{v}_i)$ to conclude 
that $(I-P_k)v^*_{g_0}P_k=0$ for all $k \in \hat{H}$. This shows that $v_{g_0}$ 
commutes with $u_z$ for all $z \in H$ i.e. $\beta_z(v_{g_0})=v_{g_0}$ for 
all $z \in H$.

\vsp 
Fix any $z \in H$. By taking action of $\beta_z$ on both sides of the relation $v_{g_0}\beta_{g_0}(v_k)v_{g_0}^*=\clj \tilde{v}_k \clj$, we have 
$v_{g_0}v_kv^*_{g_0} = \bar{z}^2  \clj \tilde{v}_k \clj  = \bar{z}^2 v_{g_0}\beta_{g_0}(v_k)v^*_{g_0}$ since $u_z$ commutes with $\clj$ and $v_{g_0}$. Thus $z^2=1$ 
since $\sum_k v_kv_k^*=1$. Thus $H \subseteq \{-1,1 \}$.   

\vsp 
We set an anti-linear $*$-automorphism $\clj_{g_0}: \clo_d \otimes \tilde{\clo}_d \raro  \clo_d \otimes \tilde{\clo}_d$ 
defined by 
$$\clj_{g_0}(s_Is^*_J \otimes \tilde{s}_{I'}\tilde{s}^*_{J'}) = \beta_{g_0}(s_{I'}s^*_{J'}) \otimes \beta_{g_0}(\tilde{s}_I\tilde{s}^*_{J})$$
for $|I|,|J|,|I'|,|J'|  <  \infty$ by extending anti-linearly. 

\vsp 
We say a state $\psi$ on $\clo_d \otimes \tilde{\clo}_d$ 
is {\it reflection positive with a twist $g_0$} if $\psi(\clj_{g_0}(x)x) \ge 0$ for all $x \in \clo_d$ and equality holds,  if and only if $x=0$. Similarly a state $\omega$ on $\clb$ is called {\it reflection positivity with a twist $g_0$ } if 
$$\omega(\overline{\beta_{g_0}(\tilde{Q})}Q) \ge 0$$ 
for all $Q \in \clb_R$. Note that this notion extended to $\tilde{\clo}_d \otimes \clo_d$ is an abstract version of the concept ``reflection positivity with a twist $g_0$ '' of a state on $\clb$ introduced in [FILS] for any involution (linear or conjugate linear ) taking element from future algebra to past algebra. Such an involution are included within the abstract framework of {\it reflection positive with a twist } introduced in [FILS]. 

\vsp 
In general the hidden symmetry $v_{g_0}$ described in (42) need not be trivial and 
will play an important role in determining properties of $\omega$. Our next proposition is the simple situation that we can expect for $v_{g_0}$.    

\vsp
\begin{thm} 
Let $\omega$ be a translation-invariant, reflection symmetric with a twist $g_0 \in U_d(\IC)$ and pure state on $\clb$. Then there exists an extremal element $\psi$ in $K_{\omega}$ so that
associated Popescu's elements $(\clk,v_k,1 \le k \le d,\Omega)$ in its support projection
satisfies the relation (42) i.e.  
\be 
v_{g_0}\beta_{g_0}(v_k)v_{g_0}^*=\clj \tilde{v}_k\clj
\ee
, where $v_{g_0}$ is a unitary operator on $\clk$ commuting with the modular elements $\Delta^{1 \over 2},\clj$ and $v_{g_0}$ commutes with $P_0=[\clm_0\Omega]$. Furthermore,  the following statements are true:

\NI (a) $v_{g_0}$ is self-adjoint,  if and only if $g_0=\bar{g_0}$; 

\NI (b) If $\omega$ is also reflection symmetric then $H \subseteq \{-1,1\}$ and $v_{g_0}$ commutes with $\{u_z:z \in H \}$; 

\NI (c) $\psi$ is reflection positive with twist $g_0$ on $\pi(\tilde{\clo_d} \otimes \clo_d)$,  if and only if $v_{g_0}$ in (43) is equal to $1$ i.e. we have 
$$\clj \tilde{v}_k \clj = \beta_{g_0}(v_k) $$ 
for all $1 \le k \le d$. 

\NI (d) $\omega$ is reflection positive with twist $g_0$ on $\clb$,  if and only if 
$$\clj \tilde{v}_I\tilde{v}_J^* \clj = \beta_{g_0}(v_Iv_J^*)$$ 
for all $|I|=|J| < \infty $. In such a case $v_{g_0}P_0=P_0$  
and $\tilde{\tau}(y)=\clj \tau(\clj y \clj )\clj$ for $y \in \clm_0' \subseteq \clb(\clk_0)$,  where we recall $P_0=[\clm_0\Omega]$ and $\clk_0$ is the 
Hilbert subspace $P_0$ of $\clk=[\clm\Omega]$. 

\NI (e) $\Delta=I_{\clk}$,  if and only if 
$$v_{g_0}\beta_{g_0}(v_k)v^*_{g_0} = v_k^*,\; 1 \le k \le d$$ 
In such a case $H$ is trivial and $\clm$ is finite type-I and spatial correlation functions 
of $\omega$ decays exponentially. Further if $\omega$ is reflection positive with twist $g_0$ on $\clb$, then $v_{g_0}=1$.      
\end{thm} 

\vsp
\begin{proof} We have already proved existence of an extremal element $\psi$ in $K_{\omega}$ and Popescu's elements on its support projection satisfying the equality 
(42) and also (a)-(b) in the text that followed after the completion of Theorem 3.4. 
We need to prove statements (c), (d) and (e). 

\vsp 
The state $\omega$ being pure, we recall from Proposition 3.3 (d) that $P\pi(\clo_d)''P=\clm$ and $P\pi(\tilde{\clo}_d)''P = \tilde{\clm} \subseteq \clm'$ 
( we do not need equality here ) and  $P=E\tilde{E}$. Thus for any $x \in \clo_d$ 
, we may write 
$$\psi(\clj_{g_0}(x) x )$$ 
$$=  \langle \Omega, \pi(\clj_{g_0}(x)) \pi(x) \Omega \rangle  $$
$$=  \langle \Omega, P \pi(\clj_{g_0}(x)) P\pi(x)P \Omega \rangle  $$
$$=  \langle \Omega, \clj_{g_0} P\pi(x)P \clj_{g_0} P\pi(x)P \Omega \rangle $$
, where we have used equality $\pi(\clj_{g_0}(x)) = \clj_{g_0} \pi(x) \clj_{g_0}$ from Theorem 3.4. If $v_{g_0}=1$ i.e. $\clj_{g_0}=v_{g_0}\clj=\clj$ on $\clk$ and thus we have 
$ \langle \Omega,\clj P\pi(x)P \clj P\pi(x)P \Omega \rangle  \ge 0$ by the self-dual property of 
Tomita's positive cone $\overline{ \{ \clj a \clj a \Omega: a \in \clm \} }$ [BR1]. Thus $\psi$ is a reflection positive map on $\pi(\tilde{\clo}_d \otimes \clo_d)$.  

\vsp 
Conversely, if $\psi$ is reflection positive on $\pi(\tilde{\clo}_d \otimes  \clo_d)''$ we have 
$ \langle \Omega,a \clj_{g_0} a \clj_{g_0} \Omega \rangle  \ge 0$,  where $a \in \clm = P\pi(\clo_d)''P$. Since $v_{g_0}$ commutes with $\clj$ and $\Delta^{1 \over 2}$ we may rewrite $ \langle \Omega, a v_{g_0} \clj a \Omega \rangle  = 
 \langle a^*\Omega, v_{g_0} \Delta^{1 \over 2} a^*\Omega \rangle   \ge 0$ i.e. $v_{g_0}\Delta^{1 \over 2}$ is a non-negative operator. 
Since $\Delta^{-{1 \over 2}}$ is also a non-negative operator commuting with $v_{g_0}\Delta^{1 \over 2}$, 
we conclude that $v_{g_0}$ is a non-negative operator. $v_{g_0}$ being unitary we conclude that $v_{g_0}=1$. This completes the proof of (c).  

\vsp 
The statement (d) also follows along the same route that of (c) replacing the role of $\clm$ and $\tilde{\clm}$ by $\clm_0$ and $\tilde{\clm}_0$ respectively with state $\omega=\psi_{|}\clb$.   

\vsp 
We will deal with the non-trivial part of (e). Let $v_{g_0}\beta_{g_0}(v_k)v^*_{g_0}=v_k^*$ for all $1 \le k \le d$. For all $1 \le k \le d$, by (43) and we have 
$v_k^* = \clj \tilde{v}_k \clj $. By our construction given in (19), following [BJKW],  
we have  
$\clj v_Iv^*_J \clj = \overline{\clj \sigma_{i \over 2}(v_Iv^*_J)\clj}$. Thus 
in particular,  we have
$\clj v_Iv_J^* \omega = \clj \Delta^{1 \over 2} v_Iv_j^*\Omega$. Since 
$[\clm \Omega]=\clk$, we conclude that $\clj = \clj \Delta^{1 \over 2}$. Thus $\Delta=I_{\clk}$. 

\vsp 
In general $\omega$ being a pure state $\clm$ is either a type-I or type-III factor [Mo1, Theorem 1.1]. Thus we conclude that $\clm$ is a finite type-I factor if $\Delta=1$ ( i.e. $\phi$ is a tracial state on $\clm$  ). This completes the proof of the first part of (e).  

\vsp
The last part of (e) is rather elementary. By Proposition 2.5 (a),  we have 
$$\omega(X_l \theta_k(X_r)) = \phi(\clj x_l \clj \tau_k(x_r))$$
for any $X_l \in \clb_L $ and $X_r \in \clb_R$, where $\clj x_l \clj = PX_lP$ 
and $x_r=PX_rP$ for some $x_l,x_r \in \clm$. 
By Power's criteria (1) we get in particular,  
$$<<x,\tau^n(y)>> \raro <<x,I_{\clk}>><<I_{\clk},y>>$$ as $n \raro \infty$. 
Since $\clm_0=P \pi(\clb_R)''P$ and $\clm_0$ is dense in $\clk^{1 \over 2}_0$, 
conclude that for all $f,g \in \clk^{1 \over 2}_0$,  we have 
$$<<f,T^ng>> \raro <<f,I_{\clk}>><<I_{\clk},g>>$$ 
as $n \raro \infty$ by a standard approximation method. In particular,  the point 
spectrum of the self-adjoint contractive operator $T$, defined by 
$Tx=\tau(x)$ on the KMS Hilbert space $\clk^{1 \over 2}_0$, in the unit 
circle is trivial i.e. $\{ z \in S^1: Tf=zf,\;\mbox{for some non zero }\; f \in \clk \}$ 
is the trivial set $\{ 1 \}$. 

\vsp 
Furthermore,  $T$ being a contractive matrix on a finite dimensional 
Hilbert space, the spectral radius of $T-|\Omega \rangle\langle \Omega|$ 
is $\alpha$ for some $\alpha < 1$. Now we use Proposition 3.1 once again 
for any $X_l \in \clb_L $ and $X_r \in \clb_R$ to verify the following
$$e^{\delta k }|\omega(X_l \theta_k(X_r)) -\omega(X_l)\omega(X_r)| $$
$$= e^{\delta k}|\phi(\clj x_l \clj \tau_k(x_r))- \phi(x_l)\phi(x_r)| \raro 0$$
as $k \raro \infty $ for any $\delta >0$ so that $e^{\delta}\alpha < 1$,  where 
$\clj x_l \clj = PX_lP$ and $x_r=PX_rP$ for some $x_l,x_r \in \clm$. As $\alpha < 1$ 
such a $\delta >0$ exists. 

\vsp 
The von-Neumann algebra $\clm$ being a finite factor, $\omega_R$ is a type-I factor state of $\clb_R$ and so $\pi(\clo_d)''=\pi(\mbox{UHF}_d)''$ by Lemma 2.3 (ii) in [Ma3] since $\omega$ is pure. Similarly we also have $\pi(\tilde{\clo}_d)''=\pi(\tilde{\mbox{UHF}}_d)''$. This completes the proof for (e) as the last statement follows from (c) since we have shown 
$\pi(\clo_d)''=\pi(\mbox{UHF}_d)''$ and $\pi_{\psi}(\clo_d)''=\pi_{\psi}(\mbox{UHF}_d)''$.
\end{proof}

\section{ Translation invariant twisted reflection positive pure state and its split property: }

\vsp
Let $\omega$ be a translation-invariant real lattice reflection-symmetric with a twist $g_0$ pure state on $\clb$. We fix an extremal element $\psi \in K_{\omega}$ so that $\bar{\psi}=\tilde{\psi}^{g_0}= \psi \beta_{\zeta}$ and consider the 
Popescu's elements $(\clk,\clm,v_i,\Omega)$ as in Theorem 3.5. $P$ being the support projection of a factor state $\psi$ we 
have $\clm = P\pi(\clo_d)''P = \{ v_k,v^*_k: 1 \le k \le d \}''$ ( Proposition 2.4 in [Mo2] ). So the dual Popescu's elements 
$(\clk,\clm',\tilde{v}_k,1 \le k \le d, \Omega)$ satisfy the relation 
$$v_{g_0}\beta_{g_0}(v_k)v^*_{g_0}= \clj \tilde{v}_k \clj$$ 
for all $1 \le k \le d$.

\vsp
We quickly recall as $\clm_0$ is the $\{\beta_z:\;z \in H \}$ invariant elements of $\clm ( =P\pi(\clo_d)''P )$, 
the normal conditional expectation $a \raro \int_{z \in H}\beta_z(a)dz$ from $\clm$ onto $\clm_0$ preserves the 
faithful normal state $\phi$. So by Takesaki's theorem [Ta] modular group associated with $\phi$ preserves 
$\clm_0$. Further since $\beta_z(\tau(a))=\tau(\beta_z(a))$ for all $x \in \clm$, the restriction of the 
completely positive map $\tau(a)=\sum_k v_kav_k^*$ to $\clm_0$ is a well defined map on $\clm_0$. Hence the 
completely positive map $\tau(a)=\sum_k v_kav_k^*$ on $\clm_0$ is also KMS symmetric modulo a unitary conjugation by $v_{g_0}$ 
i.e. $$ \langle  \langle a,\tau(b) \rangle  \rangle = \langle  \langle \tau_{v_{g_0}}(a),b \rangle  \rangle$$
, where 
$$ \langle  \langle a,b \rangle  \rangle 
=\langle \Omega,\clj a \clj b \Omega \rangle =\langle a \Omega, \Delta^{1 \over 2}b\Omega \rangle$$ 
for all $a,b \in \clm_0$ and $(\sigma_t(a)=\Delta^{it}x\Delta^{-it})$ 
is the modular automorphism group on $\clm_0$ associated with $\phi_0$ and $[\clm_0\Omega]=\clk_0$, where $\clk_0$ is the Hilbert subspace of $\clk$ equal to the range of the projection $P_0$ and 
$$\tau_{v_{g_0}}(a)=v^*_{g_0}\tau(v_{g_0}av^*_{g_0})v_{g_0}$$
for all $a \in \clm_0$. Thus $\tau_{v_{g_0}}=\tau$ on $\clm_0$,  if and only if $\omega$ is reflection positive on $\clb$ with twist $g_0$ (Theorem 3.5). However,  the inclusion $\clm_0 \subseteq \clm$ need not be an equality in general unless $H$ is trivial. 

\vsp
We now fix a translation-invariant real lattice reflection-symmetric pure state $\omega$ which is also reflection positive with a 
twist $g_0$ on $\clb$ and explore KMS-symmetric property of $(\clm_0,\tau,\phi)$ and the extended Tomita's conjugation 
operator $\clj_{g_0}$ on $\clh \otimes_{\clk} \tilde{\clh}$ defined in Theorem 3.5 to study the 
relation between split property and exponential decaying property of spatial correlation functions of $\omega$. 

\vsp
For any fixed $n \ge 1$ let $Q \in \pi(\clb_{[-n+1,n]})$. We write
\be 
Q = \sum_{|I|=|J|=|I'|=|J'|=n} q(I',I|J',J)\beta_{g_0}(\tilde{S}_{I'}\tilde{S}^*_{J'})S_IS^*_J 
\ee
Since the elements $\beta_{g_0}(\tilde{S}_{I'}\tilde{S}^*_{J'})S^*_IS_J: |I|=|J|=|I'|=|J'|=n$ form a linear independent basis for $\pi(\clb_{[-n+1,n]})$, for an element $Q \in \clb_{[-n+1,n]}$ we find a unique representation of $Q$ given in (44) with 
$q$, where $q$ is the matrix $q=((q(I',I|J',J) ))$ of order $d^{2n} \times d^{2n}$.
The map $Q \raro q$ is an automorphism between two finite dimensional algebras. Thus 
the operator norm of $Q$ is equal to the matrix norm of $q$. Further $Q$ is positive 
, if and only if $q$ is so. 

\vsp 
We consider the linear map $L:M_{d^{2n} \times d^{2n}}(\IC) \raro M_{d^{2n} \times d^{2n}}(\IC)$ defined by
$$L(q)=\hat{q}$$
, where $\hat{q}=((\hat{q}(I',I|J',J) ))$ is a $d^{2n} \times d^{2n}$ matrix with 
$$\hat{q}(I',I|J',J) =  q(I',J'|I,J) $$ 
Note by our definition we can as well write   
\be 
Q= \sum_{|I|=|J|=|I'|=|J'|=n} \hat{q}(I',J'|I,J)\beta_{g_0}(\tilde{S}_{I'}\tilde{S}^*_{J'})S_IS^*_J 
\ee
We also set 
\be 
\hat{Q} = \sum_{|I|=|J|=|I'|=|J'|=n} \hat{q}(I',I|J',J)\beta_{g_0}(\tilde{S}_{I'}\tilde{S}^*_{J'})S_IS^*_J 
\ee

\vsp 
Though the map $L$ is not unit preserving, we have the following property.

\vsp
\begin{pro} 
The map $Q \raro \hat{Q}$ is positive i.e. $\hat{q} \ge 0$ whenever $q \ge 0$; 
\end{pro}

\vsp
\begin{proof} 
We verify the following simple steps by taking transpose $q \raro q_e^t$ as with respect to the orthonormal basis $(e_i)$ described as in (6):
$$((\hat{q}(I',J'|I,J)))^t$$
$$=((\hat{q}(I,J|I',J')))$$
$$=((q(I,I'|J,J'))$$
$$=U((q(I',I|J',J))))U^*$$
, where $U$ is the unitary matrix that takes orthonormal basis vector $e_{I'} \otimes e_{I}$ to
$e_{I} \otimes e_{I'}$,  where $e_{I'}= e_{i'_1} \otimes e_{i'_2}...\otimes e_{i'_n} $ and $e_I=e_{i_1} \otimes e_{i_2} \otimes ..\otimes e_{i_n}$ 
Since the transpose map $q \raro q^t$ is positive ( $(q^*q\hat{)}=\hat{q}\hat{q^*}=\hat{q}(\hat{q})^*)$, we conclude that $\hat{q}$ is positive whenever $q$ is so. 
\end{proof}

\vsp
\begin{pro}
Let $\omega$ be a translation-invariant real lattice reflection-symmetric with twist $g_0$ pure state on $\otimes_{\IZ}M_d$ . Then there exists an extremal point $\psi \in K_{\omega}$ so that
$\psi \beta_{\zeta}=\tilde{\psi}^{g_0}=\bar{\psi}$,  where $\zeta \in \{1, \mbox{exp}^{i \pi \over 2} \}$ and the associated Popescu's elements $(\clh,S_k,\;1 \le k \le d,\Omega)$ and $(\clh,\tilde{S}_k,\;1 \le k \le d,\Omega)$ described in Proposition 2.5 satisfy the following:

\NI (a) For any $n \ge 1$ and $Q \in \pi(\clb_{[-n+1,n]})$ we write
$$Q=\sum_{|I'|=|J'|=|I|=|J|=n} \hat{q}(I',J'|I,J)\beta_{g_0}(\tilde{S}^*_{I'}\tilde{S}^*_{J'})S^*_IS_J$$ and
and for $k \ge 1$  
$$\hat{\theta}_{2k}(Q)= \sum_{|I|=|J|=|I'|=|J'|=n} \hat{q}(I',J'|I,J)
\beta_{g_0}(\tilde{\Lambda}^k(\tilde{S}_{I'}\tilde{S}^*_{J'}))\Lambda^{k}(S_IS^*_J),$$
where $\tilde{\Lambda}(X)=\sum_{1 \le i \le d} \pi(\tilde{s}_i)X\pi(\tilde{s}^*_i)$. Then 
$$\hat{\theta}_{2k}(Q) \in \clb_{(-\infty,-k] \bigcup [k+1,\infty)}$$ 

\NI (b) $Q=\clj_{g_0} Q \clj_{g_0}$,  if and only if $\hat{q}(I',J'|I,J) = \overline{\hat{q}(I,J|I',J')};$

\NI (c) If the matrix $\hat{q}=(( \hat{q}(I',J'|I,J) ))$ is non-negative then there exists a matrix $\hat{b}=(( \hat{b}(I',J'|I,J) ))$ so that $\hat{q}=(\hat{b})^*\hat{b}$ and then
$$PQP=\sum_{|K|=|K'|=n} \clj_v x_{K,K'} \clj_v x_{K,K'}$$
, where $x_{K,K'} = \sum_{I,J:\;|I|=|J|=n} \hat{b}(K,K'|I,J)v_Iv^*_J \in \clm_0$
\vsp
\NI (d) In such a case i.e. if $Q=\clj_{g_0} Q \clj_{g_0}$ the following holds:

\NI (i) $\omega(Q)=\sum_{K,K':|K|=|K'|=n} \phi(\clj_v x_{K,K'} \clj_v x_{K,K'})$

\NI (ii) $\omega(\hat{\theta}_{2k}(Q))=\sum_{K,K':|K|=|K'|=n}\phi(\clj_v x_{K,K'} \clj_v \tau_{2k}(x_{K,K'})).$

\end{pro} 

\vsp
\begin{proof} 
Since the elements $\beta_{g_0}(\tilde{S}_{I'}\tilde{S}^*_{J'})S^*_IS_J: |I|=|J|=|I'|=|J'|=n$ form a linear independent basis for $\pi(\clb_{[-n+1,n]})$ and thus for an element $Q \in \clb_{[-n+1,n]}$, the representation of $Q$ given in (44) with $q=((q(I',I|J',J) ))$
is unique and so also $\hat{q}$ defined in (45). 

\vsp 
The endomorphism $\Lambda$ is the right translation on $\pi(\clb_R)''$ fixing all elements in $\pi(\clb_L)''$ and the endomorphism $\tilde{\Lambda}$ is left translation on $\pi(\clb_L)''$ fixing all elements in $\pi(\clb_R)''$. Thus (a) follows. 

\vsp 
The statement (b) is also a simple consequence of unique representation of $Q$ given in (44) and the relation $\clj_{g_0} \beta_{g_0}(\tilde{S}_{I'}\tilde{S}^*_{J'})S_IS^*_J \clj_{g_0} = S_{I'}S^*_{J'} \beta_{g_0}(\tilde{S}_I\tilde{S}^*_J) $. 

\vsp 
For (c) we write 
$$Q=\sum_{|K|=|K'|=n} \clj_{g_0}(Q_{K,K'}) \clj_{g_0} Q_{K,K'}$$ 
, where $Q_{K.K'}= \sum_{I,J:\;|I|=|J|=n} \hat{b}(K,K'|I,J)S_IS^*_J$. The state 
$\omega$ being pure,  we have by Theorem 3.6 in [Mo2] that 
$P=E \tilde{E}$,  where $E$ and $\tilde{E}$ are support projection of $\psi$ in $\pi(\clo_d)''$
and $\pi(\tilde{\clo}_d)''$ respectively. So for any $X \in \pi(\clo_d)''$ and $Y \in \pi(\tilde{\clo}_d)''$
we have $PXYP=\tilde{E}EXY\tilde{E}E= \tilde{E}EYE\tilde{E}X\tilde{E}E=PXPYP$. Thus (c) follows as $\omega(Q)
= \phi(PQP)$ by Theorem 3.5 as $\omega$ is lattice symmetry with twist $g_0$. 
For (d) we use (a) and (c). This completes the proof. 
\end{proof} 

\vsp
\begin{pro} 
Let $\omega$, a translation-invariant pure state on $\clb$, be in detailed balance and 
reflection positive with a twist $g_0$. Then the following are equivalent:

\NI (a) two-point spatial correlation functions of $\omega$ decay exponentially;

\NI (b) The spectrum of $T-|I_{\clk} \rangle\rangle\langle\langle I_{\clk}|$ is a subset of $[-\alpha,\alpha]$ for some $0 \le \alpha < 1$
, where $T$ is the self-adjoint contractive operator defined by
$$Ta\Omega= \tau(a)\Omega,\;\;a \in \clm_0$$
on the KMS-Hilbert space $\clk^{1 \over 2}_0$ with inner product $ \langle  \langle a,b \rangle  \rangle = 
\langle \Omega, \clj a \clj b \Omega\rangle.$
\end{pro} 

\vsp
\begin{proof} By Proposition 2.5 (d) we have $P\pi(\clb_R)''P = \clm_0$ and 
$P\pi(\clb_L)''P = \tilde{\clm}_0 = \clj \clm_0 \clj$as $\tilde{\clm}=\clm'=\clj 
\clm \clj$.  

\vsp 
Since $T^k a \Omega = \tau^k(a)\Omega$ for $a \in \clm_0$ and
for any $L \in \clb_L$ and $R \in \clb_R$, by Theorem 3.5,  we have   
$\omega(L \theta^k(R)) = \phi(\clj b \clj \tau^k(a)) =  \langle  \langle  b,T^ka  \rangle  \rangle, $
where $a=P\pi(R)P$ and $b=\clj P\pi(L)P \clj$ are elements in $\clm_0$. 

\vsp 
We conclude that (a) holds,  if and only if for some $\delta > 0$,  we have 
$$e^{k \delta}| \langle \langle f,T^kg \rangle\rangle - \langle\langle f,I_{\clk} \rangle\rangle\langle\langle I_{\clk},g \rangle\rangle | \raro 0$$ 
as $k \raro \infty$ for al vectors $f,g$ in the dense subset $\clm_0$ 
of the KMS Hilbert space $\clk^{1 \over 2}_0$.

\vsp
That (b) implies (a) is now obvious since $e^{k\delta}\alpha^k=(e^{\delta}\alpha)^k \raro 0$ whenever we choose
a $\delta  >  0$ so that $e^{\delta}\alpha < 1$ since $0 \le \alpha < 1 $.

\vsp
For the converse suppose that (a) holds and $T^2-|I_{\clk} \rangle\rangle\langle\langle I_{\clk}|$ is not bounded away from $1$. Since
$T^2-|I_{\clk} \rangle\rangle\langle\langle I_{\clk}|$ is a positive self-adjoint contractive operator, for each $n \ge 1$, we find a unit 
vector $f_n$ in the Hilbert space so that $E_{[1-1/n,1]}f_n=f_n$ and $f_n \in \cld$, where $s \raro E_{[s,1]}$ 
is the spectral family of the positive self-adjoint operator $T^2-|I_{\clk} \rangle\rangle\langle\langle I_{\clk}|$ and in order to ensure 
$f_n \in \cld$ we also note that $E_{[s,1]}\cld=\{E_{[s,1]}f:\;\;f \in \cld \}$ is dense in $E_{[s,1]}$ 
for any $0 \le s \le 1$.  

\vsp
Thus by exponential decay there exists a $\delta > 0$ so that
$$e^{2k \delta} (1-{1 \over n})^k \le e^{2k \delta} \int_{[0,1]}s^k \langle\langle f_n,dE_sf_n \rangle\rangle =
e^{2k \delta } \langle\langle f_n,[T^{2k}-|I_{\clk} \rangle\rangle\langle\langle I_{\clk}|]f_n \rangle\rangle  \raro 0$$
as $k \raro \infty$ for each $n \ge 1$. Hence $e^{2\delta}(1-{1 \over n}) < 1$. Since $n$ is any
integer, we have $e^{2 \delta} \le 1$. This contradicts that $\delta > 0$.
This completes the proof.
\end{proof} 

\vsp
\begin{proof} (of Theorem 1.3) It is enough if we verify (3) for every elements $Q \in \pi(\clb_{loc})$, $Q=Q^*$ for split property. We fix any $n \ge 1$ and a self-adjoint element $Q \in \clb_{[-n+1,n]}$ i.e. 
$q$ is self-adjoint. The matrix $q$ being symmetric and we can write $q=q_+-q_-$,  where $q_+$ and $q_-$ are the unique
non-negative matrices contributing its positive and negative parts of $q$. 
Hence $||q_+|| \le ||q||$ and $||q_-|| \le ||q||$. 

\vsp 
We set as before $\hat{q}=L(q)$ and similarly $\hat{q}_+=L(q_+) \ge 0$ and $\hat{q}_-= L(q_-) \ge 0$ by Proposition 4.1 (a). We write as in Proposition 4.2
$$Q=\sum_{|I'|=|J'|=|I|=|J|=n}\hat{q}(I',J'|I,J)\beta_{g_0}(\tilde{S}^*_{I'}\tilde{S}_{J'})S^*_IS_J$$ 
and similarly $Q_+$ and $Q_-$.  

\vsp 
We recall 
$$\hat{\theta}_{2k}(Q)= \sum_{|I|=|J|=|I'|=|J'|=n} \hat{q}(I',J'|I,J)\beta_{g_0}(\tilde{\Lambda}^k(\tilde{S}_{I'}\tilde{S}^*_{J'}))
\Lambda^k(S_IS^*_J)$$ which is an element in $\clb_{(-\infty,-k] \bigcup [k+1,\infty)}$ and by Proposition
4.2 (d) $$\omega(\hat{\theta}_{2k}(Q))=\sum_{|K|=|K'|=n}\phi(\clj x_{K,K'} \clj \tau^{2k}(x_{K,K'}))$$ provided
$\hat{q}=(\hat{q}(I',J'|I,J)$ is positive (i.e. by Proposition 4.1 if $Q$ positive ), where $$PQP=\sum_{|K|=|K'|=n}\clj x_{K,K'}\clj x_{K,K'}$$ 
and 
$$x_{K,K'}=
\sum_{I,J}\hat{b}(K,K'|I,J)v_Iv^*_J$$ 
and $\hat{q}=(\hat{b})^*\hat{b}$. Thus in such a case,  we have by Proposition 4.2 (d) that
$$|\omega(\hat{\theta}_{2k}(Q))-\omega_L\otimes \omega_R(\hat{\theta}_{2k}(Q))| =
\sum_{|K|=|K'|=n}\langle \Omega, \clj x_{K,K'} \clj (\tau^{2k}-\phi)(x_{K,K'}) \Omega \rangle$$
\be 
=\sum_{|K|=|K'|=n} \langle  \langle  x_{K,K'}, ( T - |I_{\clk} 
\rangle\rangle\langle\langle I_{\clk}| )^{2k} x_{K,K'} \rangle  \rangle  
\ee
$$\le \alpha^{2k} \sum_{|K|=|K'|=n} \langle  \langle  x_{K,K'}, x_{K,K'} \rangle  \rangle  $$
provided $||T-|I_{\clk} \rangle\rangle \langle \langle I_{\clk}||| \le \alpha$ and so 
$$\le \alpha^{2k}\omega(Q) \le \alpha^{2k}||q|| $$

\vsp
Hence for an arbitrary $Q$, we have
$$ |\omega(\hat{\theta}_k(Q))-\omega_L\otimes \omega_R(\hat{\theta}_k(Q))| \le \alpha^{2k}(||
q_+||+||q_-||) \le 2 \alpha^{2k}||q|| = 2\alpha^{2k}||Q||$$ 
This completes if part of the proof by Proposition 4.3.

\vsp 
For the converse we will show now in the following text that split property (3) also implies that 
$||(T-|I_{\clk} \rangle\rangle\langle\langle I_{\clk}|)^{2k}|| \raro 0$ as $k \raro \infty$. 

\vsp 
We consider the Hilbert space $l^2(\clk^{1 \over 2}_0)$ and vectors   
$$X_b = \oplus_{|K|=|K'|=n} x^b_{K,K'} \in l^2(\clk^{1 \over 2}_0),$$ 
where $x^b_{K,K'}=\sum_{|I|=|J|=n } \hat{b}(K,K'|I,J)v_Iv_J^*$ and 
$\langle\langle X_b, X_b \rangle \rangle =\omega(Q) \le ||Q||$.
Since $\hat{q}=(\hat{b})^*\;\hat{b}=\hat{b^*}\;\hat{b}=(bb^*\hat{)}$, we get $bb^*=\hat{\hat{q}}=q$. Thus 
$||Q||=||b||^2 \le 1$,  if and only if $||b|| \le 1$. 

\vsp 
By the split property (3) of $\omega$ and equality in (47),  we have: for a given $\epsilon > 0$ there exists $k \ge 1$ so that   
$|\langle \langle x_b,\oplus (T-|I_{\clk} \rangle \rangle \langle \langle I_{\clk}|)^{2k} x_b \rangle \rangle| \le \epsilon$
for all $||b|| \le 1$. The map $b \raro \tilde{b}$ being affine, $\{x_b: ||b|| \le 1 \}$ is a convex subset in the unit 
ball of $l^2(\clk_0^{1 \over 2})$. Its closure is the unit ball of $l^2(\clk_0^{ 1 \over 2})$, otherwise by Hann-Banach theorem 
we will have a non zero vector $f \in l^2(\clk_0^{1 \over 2})$ orthogonal to all the vectors in the set $\{x_b: b \in \clb,\;||b|| \le 1 \}$ and hence orthogonal to all the vectors $\{x_b:b \in \clb\}$ by the linearity of the map $b \raro x_b$. However,  the vector subspace $\clm_0$ is dense in the Hilbert space $\clk_0^{1 \over 2}$ by our construction and the vectors $\{v_Iv_J^*:|I|=|J|=n,\; n \ge 1 \}$ are total in $\clk_0^{1 \over 2}$. Thus by taking all possible elementary matrix $b \in \clb$, we bring a contradiction to non zero property of $f$.  

\vsp 
This proves that $||(T-|I_{\clk} \rangle\rangle\langle\langle I_{\clk}|)||^{2k}= ||(T-|I_{\clk} \rangle\rangle\langle\langle I_{\clk}|)^{2k}|| \raro 0$ i.e. $||(T-|I_{\clk} \rangle\rangle\langle\langle I_{\clk}|)|| < 1$. Now we appeal to Proposition 4.3 
to complete the proof. \end{proof}

\vsp 
We end this section with an application of our main result Theorem 1.3 by proving a 
central limit theorem for a translation invariant state. For a historical account, on this
problem starting with the central limit theorem for stationary random fields [Bo], we refer to a recent article [Ma5]. 

\vsp 
Let $\cla$ be an $\theta$-invariant dense $*$ subalgebra of $\clb$ for which  
\be 
\sum_{ j \in \IZ} |\omega(Q_1 \theta^j(Q_2))-\omega(Q_1)\omega(Q_2)| < \infty 
\ee
for any $Q_1,Q_2 \in \cla$. For an element $Q \in \clb$, we take 
\be 
B_n(Q) = {1 \over \sqrt{2n+1}} \sum_{|j| \le n} ( \theta^j(Q) -\omega(Q) )
\ee 
for $n \ge 1$. Let $(\clh_{\pi},\pi,\Omega)$ be the GNS representation of $(\omega,\clb)$. 
Then we have 
$$\langle f, [B_n(Q_1),B_n(Q_2)]g \rangle \raro s(Q_1,Q_2)\langle f,g\rangle$$ 
for any vectors $f,g$ in the dense cyclic space of $\cla$ i.e. $f,g \in \pi(\cla)\Omega$,   
where  
\be 
s(Q_1,Q_2)= \sum_{j \in \IZ} \omega([Q_1,\theta^j(Q_2)]),
\ee
where $s(Q_1,Q_2)$ is well defined by (48) for $Q_1,Q_2 \in \cla$. Thus the formal limit
$B(Q)$ of $B_n(Q)$ gives rise to an algebra of canonical commutation relations with
respect to the degenerate symplectic form $s$ on the real vector space of self-adjoint element. This formal limit has no direct meaning as an unbounded operator. However,  we 
find its meaning interpreting it via the central limit theorem as follows [Ma5]. To that 
end we set 
\be 
\sigma_{\omega}(Q_1,Q_2)= \mbox{lim}_{n \raro \infty}\omega(B_n(Q_1)B_n(Q_2))
\ee
for $Q_1,Q_2 \in \cla$, where limit exist due to (48).    

\vsp 
\begin{defn} 
We say a translation invariant state $\omega$ of $\clb$ admits {central limit theorem } if 
\be 
\omega(e^{it B_n(Q)}) \raro e^{-{t^2 \over 2} \sigma_{\omega}(Q,Q)}
\ee  
as $n \raro \infty$ for all $Q=Q^* \in \cla$, where $\sigma_{\omega}(Q,Q)$ is positive constant depending on $Q$ satisfying (51).  
\end{defn} 

\vsp 
If $\omega$ is a split state satisfies our condition of Theorem 1.3  then condition (48) 
is satisfied, where we can take $\cla=\clb$. One simple corollary of Theorem 1.2 in [Ma5] says in such case that the state $\omega$ admits central limit theorem with $\sigma_{\omega}$ given by (51). However,  Theorem 1.2 in [Ma5] does not demand exponentially decaying property for two-point spatial correlation functions. His work also included an explicit example,  where the central limit theorem (52) holds for a non split pure state with (48) convergent. 

\vsp 
A natural question that arises here, is self-adjoint property of $T:\clk_0^{1 \over 2} \raro \clk_0^{1 \over 2}$ enough for central limit theorem (48) to hold? We defer results on this question for a possible future work.

\section{Ground states of Hamiltonian in quantum spin chain }

\vsp 
We consider [BR-II,Ru] quantum spin chain Hamiltonian in one dimensional lattice 
of the following form
\be 
H= \sum_{ k \in \IZ} \theta^k(h_0)
\ee
for $h^*_0=h_0 \in \clb_{loc}$ 
, where the formal sum gives an auto-morphism $\alpha=(\alpha_t:t \in \IR)$ via the thermodynamic limit of finite volume automorphisms $\{ \alpha^{\Lambda}_t(x)=e^{itH_{\Lambda}}xe^{-itH_{\Lambda}},\; \mbox{finite subset } 
\Lambda \uparrow \IZ \}$ whose surface energies are uniformly bounded,  where $H_{\Lambda}=\sum_{k \in \Lambda} \theta^k(h_0)$. 

\vsp 
The translation-invariant Hamiltonian $H$ is having finite range interaction and 
and $\beta-$KMS state of $(\alpha_t)$ at a given inverse positive temperature $\beta={1 \over kT}$ exists and is always unique [Ara1],[Ara2], [Ki]. Thus unique KMS-state inherits translation and other symmetry of the Hamiltonian $H$. Furthermore,  low temperature limit points of unique $\beta-$KMS states as $\beta \raro \infty$ are ground states for the Hamiltonian $H$ inheriting translation and other symmetry of the Hamiltonian $H$. It is a well known fact that the set of ground states of a translation-invariant Hamiltonian form a face in the convex set of $(\alpha_t)$-invariant states of $\clb$ and its extreme points are pure. In general the set of ground states need not be a singleton set. Ising model admits non translation-invariant ground states known as N\'{e}el state [BR vol-II]. However,  ground 
states that appear as low temperature limit of $\beta-$KMS states of a translation-invariant Hamiltonian inherit translation and other symmetry (that we have described above) 
of the Hamiltonian. In particular,  if ground state for a translation-invariant Hamiltonian model of type (53) is unique, then the ground state is a translation-invariant 
pure state. $H$ is called {\it reflection symmetric } if $\tilde{H}=H$ and {\it real} if $H^t=H$. Along the same line the unique ground state is reflection symmetry and real if $H$ is so. 

\vsp 
We recall now [DLS,FILS] if $H$ given in (53) has the following form 
\be 
-H= B + \clj_{g_0}(B) + \sum_i C_i \clj_{g_0}(C_i)
\ee 
for some $B, C_i \in \clb_R$, where $\clj_{g_0}$ is reflection with twist $g_0$ followed by conjugation 
with respect to a basis i.e. $\clj_{g_0}(B)= \overline{\beta_{g_0}(\tilde{B})}$, then KMS state at inverse positive temperature $\beta$ is refection positive with the twist $g_0$. We refer to [FILS] for details which we will cite frequently while dealing with examples. Since weak$^*$-limit of a sequence of reflection positive states with the twist $g_0$ is also a reflection positive state with twist $g_0$, weak$^*$-limit points of unique $\beta-$KMS state of $H$ as $\beta \raro \infty$ 
are also refection positive with twist $g_0$. A prime example of a Hamiltonian 
of type (53) with unique $\beta-$KMS state that satisfies reflection positive [FILS, section 3.4 ] property with a twist $g_0$ is the Heisenberg anti-ferromagnet 
iso-spin model $H_{XXX}$ with
\be 
h_0 = J (\sigma_x^0 \otimes \sigma_x^1 +\sigma_y^0 \otimes \sigma_y^1 + \sigma_z^0 \otimes \sigma_z^1), 
\ee 
where $\sigma_x^k,\sigma_y^k$ and $\sigma_z^k$ are Pauli spin matrices located at lattice site $k \in \IZ$ and $J > 0$ constant. 

\vsp 
Another prime example of $\beta-$KMS state that admits reflection positive property at inverse positive temperature $\beta$ includes anti-ferro-magnet $XY$ model namely $H_{XY}$ with 
\be 
h_0= J (\sigma_x^0 \otimes \sigma_x^1 + \sigma_y^0 \otimes \sigma_y^1)
\ee  
with $J > 0$ is well studied. In such a case any limiting state at low temperature, inherits same symmetry namely translation, lattice reflection, real, refection positive property. Hamiltonian $H_{XY}$, it is well known that the unique ground state 
[AMa] once restricted to $\clb_R$ gives a type-III factor state $\omega^{XY}_R$. Thus Theorem 1.3 says that its two-point spatial correlation function does not decay exponentially i.e. in the language of physical literature $\omega$ is {\it strongly correlated}. This results were known before and a different proof is given by Taku Matsui using altogether a different method [Ma3].  

\vsp 
On the other hand no clear picture has emerged so far about ground states of anti-ferromagnet Heisenberg $H_{XXX}$ model which serves a more realistic model for experimentally realized certain quasi-one dimensional magnetic materials. One standing conjecture by Haldane on $H_{XXX}$ model [AL,Ma3,Ma4] says that $H_{XXX}$ has a unique ground state and the state admits a mass gap with two-point spatial correlation function decaying exponentially for integer spin $s$ ( odd integer $d$,  where $d=2s+1$ ). Whereas for even values of $d$, the conjecture says that $H_{XXX}$ has a unique ground state but does not admits a mass gap and its two-point spatial correlation function does not decay exponentially (i.e. ${1 \over 2}$ odd integer spin $s$, where $d=2s+1$). A well known result due to Affleck and Lieb [AL] says that for even $d$ if $H_{XXX}$ admits unique ground state then two-point spatial correlation functions do not decay exponentially. We refer [AL] for finer details on this point and many other related rigorous results. 

\vsp 
If ground state of anti-ferromagnetic $H_{XXX}$ model is unique for ${1 \over 2}$ odd integer spin degrees of freedom i.e. $d=2s+1$ is an even integer, then our main result says that two-point spatial correlation functions of the ground state do not decay exponentially as $\pi(\clb_R)''$ can not be a type-I factor by Theorem 1.3 in [Ma3]. Now Theorem 2 in [NaS] 
says that the ground state if unique can not have mass gap. In [Mo3] we investigate this issue further by studying additional continuous symmetry of the state in the light of a well known general result [Wa] on ergodic actions of a compact group on von-Neumann algebras.    

\bigskip
{\centerline {\bf REFERENCES}}

\begin{itemize} 
\bigskip
\item{[Ac]} Accardi, L. : A non-commutative Markov property, (in Russian), Functional.  anal. i Prilozen 9, 1-8 (1975).

\item{[AcC]} Accardi, Luigi; Cecchini, Carlo: Conditional expectations in von Neumann algebras and a theorem of Takesaki.
J. Funct. Anal. 45 (1982), no. 2, 245–273. 

\item{[AKLT]} Affleck, L.; Kennedy, T.; Lieb, E.H.; Tasaki, H.: Valence Bond States in Isotropic Quantum Antiferromagnets, Commun. Math. Phys. 115, 477-528 (1988). 

\item{[AL]} Affleck, L.; Lieb, E.H.: A Proof of Part of Haldane's Conjecture on Spin Chains, Lett. Math. Phys, 12,
57-69 (1986). 

\item{[Ara1]} Araki, H.:  Gibbs states of a one dimensional quantum lattice. Comm. Math. Phys. 14 120-157 (1969). 

\item{[Ara2]} Araki, H.: On uniqueness of KMS-states of one-dimensional quantum lattice systems, Comm. Maths. Phys. 44, 1-7 (1975).

\item{[AMa]} Araki, H.; Matsui, T.: Ground states of the XY model, Commun. Math. Phys. 101, 213-245 (1985).

\item{[Bo]} Bolthauzen, E.: On the central limit theorem for stationary mixing random fields,
Ann. Prob. 4 (1982) 1047-1050

\item{[BR-I]} Bratteli, Ola,: Robinson, D.W. : Operator algebras
and quantum statistical mechanics, I, Springer 1981.

\item{[BR-II]} Bratteli, Ola,: Robinson, D.W. : Operator algebras
and quantum statistical mechanics, II, Springer 1981.

\item{[BJP]} Bratteli, Ola,; Jorgensen, Palle E.T. and Price, G.L.: 
Endomorphism of $\clb(\clh)$, Quantisation, nonlinear partial differential 
equations, Operator algebras, ( Cambridge, MA, 1994), 93-138, Proc. Sympos.
Pure Math 59, Amer. Math. Soc. Providence, RT 1996.

\item{[BJKW]} Bratteli, Ola,; Jorgensen, Palle E.T.; Kishimoto, Akitaka and
Werner Reinhard F.: Pure states on $\clo_d$, J.Operator Theory 43 (2000),
no-1, 97-143.    

\item{[BJ]} Bratteli, Ola ; Jorgensen, Palle E.T. :Endomorphism of $\clb(\clh)$, II, 
Finitely correlated states on $\clo_N$, J. Functional Analysis 145, 323-373 (1997). 

\item{[Cun]} Cuntz, J.: Simple $C\sp*$-algebras generated by isometries. Comm. Math. Phys. 57, 
no. 2, 173--185 (1977).  

\item{[DR]} Dagotto, E.;Rice, T.M.: Surprise on the way from one to two dimensional quantum magnets: The ladder materials, Sciences 271, 618-623 (1996)  

\item{[DLS]} Dyson, Freeman J.; Lieb, Elliott H.; Simon, B.: Barry Phase transitions in quantum spin systems with isotropic and non-isotropic interactions. J. Statistical. Phys. 18 (1978), no. 4, 335-383.

\item{[DHR]} Doplicher,S., Haag, R. and Roberts, J.: Local observables and particle statistics I, II, Comm. Math. Phys. 23 (1971) 119-230 and 35, 49–85 (1974)

\item{[DL]} Doplicher, S.; Longo, R.: Standard and split inclusions of von Neumann algebras. Invent. Math. 75, 493-536 (1984)

\item{[Ef]} Efstratios Manousakis: The spin-${1 \over 2}$ Heisenberg antiferromagnet on a square lattice and its application to the cuprous oxides, Rev. Mod. Phys. 63, 1–62 (1991

\item{[FNW1]} Fannes, M.; Nachtergaele, B.; Werner, R.: Finitely correlated states on quantum spin chains
Commun. Math. Phys. 144, 443-490(1992).

\item{[FNW2]} Fannes, M.; Nachtergaele, B.; Werner, R.: Finitely correlated pure states, J. Funct. Anal. 120, 511-
534 (1994). 

\item{[FNW3]} Fannes, M.; Nachtergaele B.; Werner, R.: Abundance of translation invariant states on quantum spin chains, Lett. Math. Phys. 25 no.3, 249-258 (1992).

\item{[FILS]} Fr\"{o}hlich, J.; Israel, R., Lieb, E.H., Simon, B.: Phase Transitions and Reflection Positivity-I 
general theory and long range lattice models, Comm. Math. Phys. 62 (1978), 1-34 

\item{[GV1]} D. Goderis, A. Verbeure and P. Vets, Non commutative central limits, Prob. Th. Related Fields 82 (1989) 527-544.

\item{[GV2]} D. Goderis and P. Vets, Central limit theorem for mixing quantum systems and the CCR-algebra of fluctuations, Commun. Math. Phys. 122 (1989) 249-265.

\item{[GV3]} D. Goderis, A. Verbeure and P. Vets, Dynamics of fluctuations for quantum lattice systems, Commun. Math. Phys. 128 (1990) 533-549

\item{[GM]} Ghosh, D.; Majumdar, C.K.: On Next ‐ Nearest ‐ Neighbor Interaction in Linear Chain. J. Math. Phys, 10, 1388 (1969). 

\item{[Hag]} Haag, R.: Local quantum physics, Fields, Particles, Algebras, Springer 1992. 

\item {[Ka]} Kadison, Richard V.: Isometries of operator algebras, Ann. Math. 54(2)(1951) 325-338.

\item{[KMSW]} Keyl, M.; Matsui, T.;Schlingemann, D.; Werner, R. F.: Entanglement Haag-duality and type properties of infinite quantum spin chains. Rev. Math. Phys. 18 (2006), no. 9, 935
-970

\item{[Ki]} Kishimoto, A.: On uniqueness of KMS-states of one-dimensional quantum lattice systems, Comm. Maths. Phys. 47, 167-170 (1976).

\item{[LSM]} Lieb, L.; Schultz,T.; Mattis, D.: Two soluble models of an anti-ferromagnetic chain Ann. Phys. (N.Y.) 16, 407-466 (1961).

\item{[Mac49]} Mackey, George W.: A theorem of Stone and von Neumann, Duke Math. J. 16 (1949), 313–326.

\item{[Ma1]} Matsui, A.: Ground states of fermions on lattices, Comm. Math. Phys. 182, no.3 723-751 (1996). 

\item{[Ma2]} Matsui, T.: A characterization of pure finitely correlated states. 
Infin. Dimens. Anal. Quantum Probab. Relat. Top. 1, no. 4, 647--661 (1998).

\item{[Ma3]} Matsui, T.: The split property and the symmetry breaking of the quantum spin chain, Comm. 
Maths. Phys vol-218, 293-416 (2001) 

\item{[Ma4]} Matsui, T.: On the absence of non-periodic ground states for the antiferromagnetic XXZ model. Comm. Math. Phys. 253 (2005), no. 3, 585-609.

\item{[Ma5]} Matsui, T.: Bosonic central limit theorem for the
one-dimensional XY Model, Reviews in Mathematical Physics, Vol. 14, No. 7 and 8 (2002) 675-700.

\item{[Mo1]} Mohari, A.: Pure inductive limit state and Kolmogorov's property. II. J. Operator Theory 72 (2014), no. 2, 387-404.

\item{[Mo2]} Mohari, A.: Translation invariant pure state on $\otimes_{k \in \IZ} \!M^{(k)}_d(\IC)$ and Haag duality. Complex Anal. Oper. Theory 8 (2014), no. 3, 745-789.

\item{[Mo3]} Mohari, A.: $SU(2)$ spontaneous symmetry breaking in ground state of quantum spin chain, Preprint. 
 
\item{[Na]} Nachtergaele, B. Quantum Spin Systems after DLS 1978, `` Spectral Theory and Mathematical Physics: A Festschrift in Honor of Barry Simon's 60th Birthday '' Fritz Gesztesy et al. (Eds), 
Proceedings of Symposia in Pure Mathematics, Vol 76, part 1, pp 47--68, AMS, 2007. 

\item{[NaS]} Nachtergaele, B; Sims, R. : Lieb-Robinson bounds and the exponential clustering theorem, Comm. Math. Phys. 265, 119-130 (2006). 

\item{[OP]} Ohya, M., Petz, D.: Quantum entropy and its use, Text
and monograph in physics, Springer-Verlag

\item{[Po]} Popescu, G.: Isometric dilations for infinite sequences of non-commutating operators, Trans. Amer. Math.
Soc. 316 no-2, 523-536 (1989)

\item{[Pow]} Powers, Robert T.: Representations of uniformly hyper-finite algebras and their associated von Neumann. rings, Annals of Math. 86 (1967), 138-171.

\item{[Ru]} Ruelle, D. : Statistical Mechanics, Benjamin, New York-Amsterdam (1969) . 

\item{[Si]} Simon, B.: The statistical mechanics of lattice gases. Vol. I. Princeton Series in Physics. Princeton University Press, Princeton, NJ, 1993. 

\item{[Sa]} Sakai, S. : Operator algebras in dynamical systems. The theory of unbounded derivations in $C\sp *$-algebras. 
Encyclopedia of Mathematics and its Applications, 41. Cambridge University Press, Cambridge, 1991. 

\item{[So]} St\o rmer E.: On projection maps of von Neumann algebras, Math. Scand. 30, 46-50 (1972). 

\item{[SW]} Summers, J. S.; Werner,R.: Maximal violation of Bell's inequalities is generic
in quantum field theory, Comm. Math. Phys. 110 (2) (1987) 247-259.

\item{[Ta]} Takesaki, M.: Conditional Expectations in von Neumann Algebras, J. Funct. Anal., 9, pp. 306-321 (1972)

\item{[Wa]} Wassermann, Antony : Ergodic actions of compact groups on operator algebras. I. General theory. Ann. of Math. (2) 130 (1989), no. 2, 273-319 

\end{itemize}

\end{document}